\documentclass[12pt]{article}

\usepackage{latexsym}
\usepackage{a4wide}
\usepackage{amscd}
\usepackage{graphics}
\usepackage{amsmath}
\usepackage{amssymb}
\usepackage{mathrsfs}
\usepackage{amsthm}
\usepackage{color}
\input xy
\xyoption{all}

\newcommand{\Z}{\mathbb{Z}}                     
\newcommand{\R}{\mathbb{R}}                     
\newcommand{\C}{\mathbb{C}}                     
\newcommand{\T}{\mathbb{T}}                     
\newcommand{\set}[2]{\left\{{#1}\mid{#2}\right\}}       
\newcommand{\ran}{\mathrm{ran\,}}		
\newcommand{\vol}{\mathrm{vol}}			

\newtheorem{thm}{\sc Theorem}      

\newtheorem*{cor*}{\sc Corollary}        

\newtheorem{lem}{\sc Lemma}            

\newtheorem*{prop*}{\sc Proposition}     

\newtheorem{rem}{\sc Remark}	    
\newtheorem*{rem*}{\sc Remark}	    	

\title{How large is the shadow of a symplectic ball?}

\author{Alberto Abbondandolo\footnote{Ruhr-Universit\"at Bochum, Fakult\"at f\"ur Mathematik, Universit\"atstrasse 150, Geb\"aude NA 4/33, D-44801 Bochum,
Germany. E-mail {\tt alberto.abbondandolo@rub.de}.} \;
 and Rostislav Matveyev\footnote{Universit\"at Leipzig, Mathematisches Institut, PF 10 09 20, D-04009 Leipzig, Germany. E-mail {\tt matveyev@math.uni-leipzig.de}.}}

\date{}

\begin{document}

\maketitle

\begin{abstract}
Consider the image of a $2n$-dimensional unit ball by a symplectic embedding into the standard symplectic vector space of dimension $2n$. Its $2k$-dimensional shadow is its orthogonal projection onto a complex subspace of real dimension $2k$. Is it true that the volume of this $2k$-dimensional shadow is at least the volume of the unit $2k$-dimensional ball? This statement is trivially true when $k=n$, and when $k=1$ it is a reformulation of Gromov's non-squeezing theorem. Therefore, this question can be considered as a middle-dimensional generalization of the non-squeezing theorem. We investigate the validity of this statement in the linear, nonlinear and perturbative setting.
\end{abstract}

\centerline{\small Mathematics Subject Classification: 37J10, 53D22, 70H15.}

\medskip

\centerline{\small Keywords: symplectic diffeomorphism, non-squeezing theorem.}

\section*{Introduction}

Let $\Omega$ be the standard symplectic form
\[
\Omega = \sum_{j=1}^n dp_j \wedge dq_j
\]
on $\R^{2n}$, the standard Euclidean space endowed with coordinates $(p_1,q_1,\dots,p_n,q_n)$. The nonsqueezing theorem of Gromov states that no symplectic diffeomorphism (i.e.\ diffeomorphism which preserves $\Omega$) can map the $2n$-dimensional ball $B_R$ of radius $R$ into the cylinder
\[
Z_S := \set{(p_1,q_1,\dots,p_n,q_n)\in \R^{2n}}{p_1^2+q_1^2 < S^2}
\]
if $S<R$ (see \cite{gro85}, and also \cite{eh89}, \cite{vit92}, \cite{hz94} for different proofs). This theorem shows that symplectic diffeomorphisms present two-dimensional rigidity phenomena (the base of the cylinder has dimension two), and not just the preservation of volume ensured by Liouville's theorem (which in the modern language just follows from the fact that, preserving $\Omega$, a symplectic diffeomorphism must preserve also $\Omega^n$, the $n$-times wedge of $\Omega$ by itself, which is a multiple of the standard volume form). Symplectic capacities are the standard tools which allow to quantify such two-dimensional rigidity phenomena (see \cite{eh89}, \cite{vit89} and \cite{hz94}).

Since symplectic diffeomorphisms preserve also the $2k$-form $\Omega^k$ for every $1\leq k \leq n$, after Gromov's result it was natural to think that there should be also middle dimensional rigidity phenomena. A possible question concerned the possibility of symplectically embedding one polydisk $\Pi := B^2_{R_1}\times \dots \times B^2_{R_n}$, where $B^2_{R_j}$ denotes the ball of radius $R_j$ in the plane associated to the coordinates $p_j$ and $q_j$, into another one $\Pi':=B^2_{R_1'}\times \dots \times B^2_{R_n'}$. If we adopt the standard convention that the radii of each of the two polydisks are increasing, Liouville's and Gromov's theorems immediately imply that if $\Pi$ can be symplectically embedded into $\Pi'$, then $R_1\cdots R_n \leq R_1' \cdots R_n'$ and $R_1\leq R_1'$. It was natural to expect other rigidity phenomena concerning other products of the $R_j$'s, but L.\ Guth recently ruled this out, by proving that there exists a constant $C(n)$ such that if $C(n)R_1\leq R_1'$ and $C(n) R_1 \cdots R_n \leq R_1' \cdots R_n'$, then $\Pi$ can be symplectically embedded into $\Pi'$ (see \cite{gut08}). See also \cite{sch05},  \cite{hut10}, \cite{hk10}, \cite{bh11}, \cite{mcd11}, \cite{ms12} and references therein for more quantitative results about the symplectic embedding problem for polydisks and other domains.

In this article, we would like to take a different point of view and to keep the ball as the domain of our symplectic embeddings. Following \cite{eg91}, we notice that Gromov's nonsqueezing theorem can be restated by saying that the two-dimensional shadow of a symplectic ball of radius $R$ in $\R^{2n}$ has area at least $\pi R^2$. More precisely, every symplectic embedding $\phi:B_R \rightarrow \R^{2n}$ satisfies the inequality
\begin{equation}
\label{area}
\mathrm{area} \bigl(P \phi (B_R) \bigr) \geq \pi R^2,
\end{equation}
where $P$ denotes the orthogonal projector onto the plane corresponding to the conjugate coordinates $p_1,q_1$. In fact, the latter statement is obviously stronger than the former. On the other hand, if the area of $A:= P \phi (B_R)$ is smaller than $\pi R^2$, then we can find a smooth area preserving diffeomorphism $\psi:A\hookrightarrow B^2_S$ for some $S<R$ (by a theorem of Moser \cite{mos65}, see also \cite[Introduction, Theorem 2]{hz94}), and the symplectic diffeomorphism $(\psi\times \mathrm{id}_{\R^{2n-2}}) \circ \phi$ maps $B_R$ into $Z_S$, thus violating the former formulation of Gromov's non-squeezing theorem. Actually, the reformulation (\ref{area}) is closer to Gromov's original proof.

In the above reformulation, the projector $P$ can be replaced by the orthogonal projector onto any complex line of $\R^{2n}\cong \C^n$, where the identification is given by 
\[
(p_1,q_1,\dots,p_n,q_n)\mapsto (p_1 + iq_1,\dots,p_n + i q_n).
\]
If one does not wish to use the complex structure of $\R^{2n}$, (\ref{area}) can be expressed using only the symplectic structure, by saying that if  $V$ is a symplectic plane in $\R^{2n}$ and $Q$ is the projector onto $V$ along the symplectic orthogonal complement of $V$, then
\begin{equation}
\label{area2}
\int_{Q \phi (B_R)} \Omega \geq \pi R^2,
\end{equation}
for every symplectic embedding $\phi:B_R \rightarrow \R^{2n}$. In fact, (\ref{area2}) follows from (\ref{area}) because the projector $Q$ is conjugated to an orthogonal projector onto a complex line by a symplectic linear automorphism of $\R^{2n}$.

Looking at the inequality (\ref{area}), it seems natural to ask whether an analogous statement holds for the volume of a higher-dimensional shadow of a symplectic ball. More precisely: if $V$ is a complex linear subspace of $\R^{2n}$ of real dimension $2k$ and $P$ is the orthogonal projector onto $V$, is it true that
\begin{equation}
\label{middle}
\vol_{2k} \bigl(P \phi (B_R) \bigr) \geq \omega_{2k} R^{2k},
\end{equation}
for every symplectic embedding $\phi:B_R \rightarrow \R^{2n}$ ? Here $\omega_{2k}$ denotes the volume of the unit $2k$-dimensional ball. Indeed, the case $k=1$ is precisely Gromov's theorem and for $k=n$ we have the equality in (\ref{middle}), by Liouville's theorem. The purely symplectic reformulation of this question, analogous to (\ref{area2}), would be asking whether 
\begin{equation}
\label{middle2}
\frac{1}{k!}  \int_{Q \phi (B_R)} \Omega^k  \geq \omega_{2k} R^{2k},
\end{equation}
when $Q$ is the projector onto a symplectic $2k$-dimensional linear subspace of $\R^{2n}$ along its symplectic orthogonal (the factor $k!$ appears because $\Omega^k$ restricts to $k!$-times the standard $2k$-volume form on every complex linear subspace of real dimension $2k$). 

The first aim of this paper is to show that (\ref{middle}) (hence also (\ref{middle2})) holds in the linear category: 

\begin{thm}[Linear non-squeezing]
\label{lns}
Let $\Phi$ be a linear symplectic automorphism of $\R^{2n}$, and let $P: \R^{2n} \rightarrow \R^{2n}$ be the orthogonal projector onto a complex linear subspace $V\subset \R^{2n}$ of real dimension $2k$, $1\leq k\leq n$. Then
\[
\vol_{2k} \bigl( P \Phi (B_R) \bigr) \geq \omega_{2k} R^{2k},
\]
and the equality holds if and only if the linear subspace $\Phi^T V$ is complex.
\end{thm} 

Here $\Phi^T$ denotes the adjoint of $\Phi$ with respect to the Euclidean product of $\R^{2n}$.
The proof is elementary but, as in the case of the standard linear non-squeezing (see \cite[Theorem 2.38]{ms98}), not completely straightforward. It is contained in Section \ref{1sec}. It is worth noticing that, unlike shadows, sections of the image of a ball by a linear symplectic automorphism $\Phi$ have small volume: if $V\subset \R^{2n}$ is a complex linear subspace of real dimension $2k$, then
\[
\vol_{2k} \bigl( V\cap \Phi (B_R) \bigr) \leq \omega_{2k} R^{2k},
\]
and the equality holds if and only if the linear subspace $\Phi^{-1} V$ is complex (see Remark \ref{sections} below).

Our second aim is to show that in the nonlinear category middle-dimensional shadows may have arbitrarily small volume:

\begin{thm}[Non-linear squeezing]
\label{squeez}
Let $P:\R^{2n}\rightarrow \R^{2n}$ be the orthogonal projection onto a complex linear subspace of $\R^{2n}$ of real dimension $2k$, with $2\leq k \leq n-1$. For every $\epsilon>0$ there exists a smooth symplectic embedding $\phi: B \rightarrow \R^{2n}$ such that
\[
\mathrm{vol}_{2k} \bigl( P \phi(B) \bigr) < \epsilon.
\]
\end{thm}

Here $B$ denotes the unit ball $B_1\subset \R^{2n}$.
The proof of this second result is based on some elementary but ingenious lemmata from the already mentioned paper of Guth \cite{gut08}. It is contained in Section \ref{2sec}.

Therefore, the middle-dimensional non-squeezing inequality (\ref{middle}) stops holding when passing from linear to nonlinear symplectic maps. However, the counterexample produced in the proof of Theorem 2 deforms the ball tremendously and it is natural to ask where the border of the validity of (\ref{middle}) lies. An interesting question in this respect seems to be: does (\ref{middle}) hold locally?

This question can be formulated in different ways. For instance, one may fix a symplectic diffeomorphism $\phi:\R^{2n}\rightarrow \R^{2n}$ and a point in its domain - without loss of generality the origin - and ask whether (\ref{middle}) holds when $R$ is small enough. 
Or one can fix a smooth path of symplectic diffeomorphisms
\[
\phi_t: \R^{2n} \rightarrow \R^{2n}, \qquad t\in [0,1],
\]
such that $\phi_0 = \Phi$, where $\Phi$ is some linear symplectic automorphism, and ask whether
\begin{equation}
\label{weak}
\mathrm{vol}_{2k} \bigl(P\phi_t (B) \bigr) \geq \omega_{2k} , \qquad \mbox{for every } 0 \leq t \leq t_0,
\end{equation}
for some positive number $t_0\leq 1$. A positive answer to this second question implies a positive answer to the first one, by considering the path of symplectic diffeomorphisms
\begin{equation}
\label{path}
\phi_t (x) := \left\{ \begin{array}{ll} \frac{1}{t} \Bigl( \phi (tx) - \phi(0) \Bigr) & \mbox{if } t\in ]0,1], \\ D\phi(0)x & \mbox{if } t=0. \end{array} \right.
\end{equation}

A first result about the first formulation of the local question is the following:

\begin{prop*}[Generic local non-squeezing]
Let $\phi:\R^{2n}\rightarrow \R^{2n}$ be a symplectic diffeomorphism and let $P$ be the orthogonal projection onto a complex subspace $V\subset \R^{2n}$ of real dimension $2k$, with $1\leq k \leq n$. Then there is an open and dense subset $U\subset \R^{2n}$ with the following property: for every $x\in U$ there exists $R_0=R_0(x)>0$ such that
\[
\vol_{2k} \bigl( P\phi (B_R(x)) \bigr) \geq \omega_{2k} R^{2k},
\]
for every $R\leq R_0$. 
\end{prop*}

We conjecture that one can actually take $U=\R^{2n}$ in the above statement and that the function $R_0$ is bounded away from zero on compact sets (this is why we do not grant the above statement the status of a theorem). The proof uses minimal submanifolds and is contained in Section 3, where we also make some general considerations about the first formulation of the local question and suggest that its positive answer might be related to the integrability of a certain ``multi-valued distribution''.

Further evidence for a positive answer to the local question is given by the next result about the second more general formulation, whose statement needs some preliminaries.

 As before, we denote the image of the orthogonal projector $P$ by $V$, which is still assumed to be a complex linear subspace of $\R^{2n}$. When $\Phi^T V$ is not a complex subspace, Theorem \ref{lns} guarantees that
\[
\mathrm{vol}_{2k} \bigl(P \Phi (B) \bigr) > \omega_{2k},
\]
so (\ref{weak}) holds for small values of $t$ just by continuity. Therefore, the question is non-trivial only when $\Phi^T V$ is a complex subspace, in particular when $\Phi=I$.

We recall that, since $\R^{2n}$ is simply connected, every symplectic path $\phi_t:\R^{2n} \rightarrow \R^{2n}$ is generated by a time-dependent Hamiltonian: There exists a smooth one-parameter family $\{H_t\}_{t\in [0,1]}$ of smooth functions on $\R^{2n}$ such that
\[
\frac{d}{dt} \phi_t(x) = X_{H_t} \bigl(\phi_t(x)\bigr), \qquad \forall (t,x) \in [0,1]\times \R^{2n},
\]
where $X_H$ is the Hamiltonian vector field associated to $H$, which is defined by the identity
\[
\imath_{X_H} \Omega = - dH.
\]
Each function $H_t$ is uniquely determined up to an additive constant. Given a $2\pi$-periodic smooth loop
\[
z: \R/2\pi \Z \rightarrow \R^{2n},
\]
we denote by
\[
E(z) := \frac{1}{2} \int_0^{2\pi} |z'(\theta)|^2\, d\theta
\]
its energy, and by
\[
A(z) := \frac{1}{2} \int_0^{2\pi} \Omega \bigl[ z(\theta), z'(\theta)\bigr]\, d\theta = \int_{\R/2\pi \Z} z^* (\Lambda), \qquad \mbox{with} \quad 
 \Lambda  := \sum_{j=1}^n p_j \, dq_j, 
 \]
the symplectic area of any disc with boundary $z$.

If $V\subset \R^{2n}$ is a complex linear subspace of complex dimension $k$, we denote by $\mathrm{Gr}_1(V)$ the Grassmannian of complex lines in $V$, equipped with the volume form
\[
\mu = \frac{1}{(k-1)!} \omega^{k-1},
\]
where $\omega$ is the standard K\"ahler form on $\mathrm{Gr}_1(V)\cong \C\mathbb{P}^{k-1}$.
Our next result is the following second order expansion for the $2k$-volume of the shadow of the image of the ball by the Hamiltonian flow $\phi_t$:

\begin{thm}[Second order expansion]
\label{expa}
Let $\{\phi_t\}_{t\in [0,1]}$ be a smooth one-parameter family of symplectic diffeomorphisms of $\R^{2n}$ into itself such that $\phi_0 = \Phi$ is linear and let $\{H_t\}_{t\in [0,1]}$ be its generating path of Hamiltonians. Let $P$ be the orthogonal projector onto a complex linear subspace $V\subset \R^{2n}$ of real dimension $2k$, with $1\leq k \leq n$, and assume that also $\Phi^T V$ is a complex subspace. Then
\[
\mathrm{vol}_{2k}\bigl(P \phi_t(B)\bigr) = \omega_{2k} + C \, t^2 + O(t^3), \qquad \mbox{for } t\rightarrow 0,
\]
the number $C$ being defined as
\[
C= C(H_0,\Phi) := \int_{\mathrm{Gr}_1(\Phi^T V)} \bigl( E(\zeta_L) - A(\zeta_L) \bigr)\, \mu (L),
\]
where $\zeta_L:\R/2\pi \Z\rightarrow \R^{2n}$ is the loop
\[
\zeta_L(\theta) = \Phi^* \bigl( (I-P) X_{H_0} \bigr) \bigl(e^{\theta J} \xi_L\bigr),
\]
$\xi_L$ being an arbitrary unit vector in $L\in \mathrm{Gr}_1(\Phi^T V)$.
\end{thm}

Here $\Phi^*( (I-P) X_{H_0})$ denotes the pull-back by $\Phi$ of the vector field $(I-P)X_{H_0}$, that is
\[
\Phi^*( (I-P) X_{H_0}) (x) := \Phi^{-1} (I-P) X_{H_0} ( \Phi x), \qquad \forall x\in \R^{2n},
\]
and $J:\R^{2n} \rightarrow \R^{2n}$ is the multiplication by $i$ in the identification $\R^{2n}\cong \C^n$.
To the best of our knowledge, this result is new also in the case $k=1$ (see \cite{lm95} for other local two-dimensional non-squeezing results). The proof is based on an elaborate series of computations, which make use of the Lie-Cartan formalism and occupy Sections \ref{4sec}, \ref{5sec}, \ref{6sec} and \ref{7sec}.

We can now use the energy-area inequality
\[
A(z) \leq E(z), \qquad \forall z\in C^{\infty}(\R/2\pi\Z,\R^{2n}),
\]
and the fact that the equality holds if and only if the loop $z$ has the form $z(\theta) = z_0 + e^{\theta J} z_1$, for some $z_0$ and $z_1$ in $\R^{2n}$, in order to deduce the following:

\begin{cor*}[Strict local non-squeezing]
Let $\phi_t$, $\Phi$, $H_t$, $P$ and $V$ be as in Theorem \ref{expa}, with $1\leq k \leq n-1$. Assume that the vector field $Z:=\Phi^* ((I-P)X_{H_0})$ does not satisfy the symmetry condition
\begin{equation}
\label{simm}
Z \left( e^{\theta J} x \right) =  \frac{1}{2} \bigl( Z(x)+Z(-x)\bigr)  + \frac{1}{2} e^{\theta J} \bigl( Z  (x) - Z(-x) \bigr) , \quad \forall x\in \partial B \cap \Phi^T V, \; \forall \theta \in \R/2\pi \Z.
\end{equation}
Then there exists $t_0>0$ such that
\begin{equation}
\label{strict}
\mathrm{vol}_{2k} \bigl(P\phi_t (B) \bigr) > \omega_{2k} , \qquad \mbox{for every } 0 < t < t_0.
\end{equation}
\end{cor*}

Therefore, the second formulation of the local non-squeezing statement holds for every middle-dimension, at least if the Hamiltonian vector field $X_{H_0}$ is not so symmetric that (\ref{simm}) holds. When (\ref{simm}) holds, then $C(H_0,\Phi)$ in Theorem \ref{expa} vanishes and we cannot expect the strict inequality (\ref{strict}) to be true (for instance, $\phi_t$ could be the identity for every $t$), but nevertheless it seems natural to conjecture that  there exists $t_0>0$ such that the weak inequality (\ref{weak}) holds. Actually, a positive answer to a conjecture of C.~Viterbo's \cite{vit00} about the relationship between the volume and the symplectic capacity of convex domains in $\R^{2n}$ would imply the above conjecture. See the end of Section \ref{7sec} for more details.

We should also mention that, since the local compactness of $\R^{2n}$ does not play any role in our arguments, all the results of this paper extend to the case of symplectic diffeomorphisms on an infinite dimensional real Hilbert space $\mathbb{H}$ which is equipped with a strong symplectic structure, that is of a continuous anti-symmetric bilinear form $\Omega:\mathbb{H} \times \mathbb{H}\rightarrow \R$ such that the anti-symmetric linear operator $J:\mathbb{H} \rightarrow \mathbb{H}$  which represents $\Omega$ with respect to the inner product, that is
\[
\Omega [u,v] = (Ju,v), \qquad \forall u,v\in \mathbb{H},
\]
is an isomorphism. In particular, the standard non-squeezing theorem which S.\ Kuksin \cite{kuk95} generalized to compact perturbations of linear operators, holds also for small perturbations of bounded linear symplectic operators on $\mathbb{H}$, again assuming that the symmetry condition (\ref{simm}) does not hold. 

\paragraph{Acknowledgements.}
This paper was influenced by stimulating discussions with Pietro Majer, Felix Schlenk and Juan Carlos \'Alvarez Paiva. We wish to thank also Ivar Ekeland and Helmut Hofer for some interesting suggestions about the local question, Marco Abate for his advice concerning the proof of Lemma \ref{red}, Camillo De Lellis, Anton Petrunin and Vitali Kapovitch for precious bibliographical suggestions. The use of the term ``shadow'' was suggested to us by Alexander Fel'shtyn. 

\section{Linear non-squeezing}
\label{1sec}

We start by recalling the formulas for the volume of the image of the ball by a linear surjection. We denote by $B$ the open unit ball about $0$ in $\R^n$. 

Let $n\geq k$ be positive integers and let $A:\R^n \rightarrow \R^k$ be linear and onto. We denote by $A^T : \R^k \rightarrow \R^n$ the adjoint of $A$ with respect to the Euclidean inner products. The linear mapping $A^T A:\R^n \rightarrow \R^n$ is symmetric and positive semi-definite, with $k$-codimensional kernel 
\[
\ker A^T A = \ker A = (\ran A^T)^{\perp}.
\]
In particular, $A^TA$ restricts to an automorphism of the $k$-dimensional space $(\ker A)^{\perp} = \ran A^T$ and, since this restriction is the composition of the two isomorphisms
\[
A|_{(\ker A)^{\perp}} : (\ker A)^{\perp} \rightarrow \R^k, \quad A^T : \R^k \rightarrow (\ker A)^{\perp},
\]
which have the same determinant, being one the adjoint of the other, we deduce that
\[
\det \left( A^T A|_{(\ker A)^{\perp}} \right) = \Bigl| \det \left( A|_{(\ker A)^{\perp}} \right)\Bigr|^2.
\]
Here, the absolute value of the determinant of linear maps between different subspaces of the same dimension is induced by the Euclidean inner product (since linear subspaces do not have a preferred orientation, the determinant is defined up to the sign).
Let $\xi_1,\dots,\xi_k$ be a basis of $(\ker A)^{\perp}$ with
\[
|\xi_1 \wedge \dots \wedge \xi_k | = 1,
\]
where the Euclidean norm of $\R^n$ is extended to multi-vectors in the standard way (in particular, $|\xi_1 \wedge \dots \wedge \xi_k |$ is the $k$-volume of the prism generated by $\xi_1,\dots,\xi_k$). Since 
\begin{equation}
\label{inpiu}
A(B) = A(B\cap (\ker A)^{\perp})= A(B\cap \ran A^T),
\end{equation} 
we find
\begin{equation}
\label{Volume}
\frac{\vol_k\bigl(A(B)\bigr)}{\omega_k}  = |A\xi_1 \wedge \cdots \wedge A\xi_k| =  |\det \left( A|_{(\ker A)^{\perp}} \right)| =
\sqrt{\det (A^TA|_{(\ker A)^{\perp}})},
\end{equation}
where $\omega_k$ denotes the $k$-volume of the unit $k$-ball.
Furthermore, the real valued function
\[
W \mapsto \bigl|\det A|_W\bigr| , \quad W \in \mathrm{Gr}_k(\R^n),
\]
where $\mathrm{Gr}_k(\R^n)$ denotes the Grassmannian of $k$-dimensional subspaces of $\R^n$, has a unique maximum at $(\ker A)^{\perp}= \ran A^T$, hence
\begin{equation}
\label{maxi}
\max_{W\in \mathrm{Gr}_k(\R^n)} \bigl|\det A|_W\bigr| = \bigl|\det A|_{\ran A^T} \bigr| .
\end{equation}

\medskip

Let $\R^{2n}$ be the $2n$-dimensional Euclidean space endowed with coordinates 
\[
(p_1,q_1,\dots,p_n,q_n),
\]
with the complex structure $J$ corresponding to the identification 
\[
(p_1,q_1,\dots, p_n,q_n) \equiv (p_1+iq_1, \dots,p_n+i q_n),
\]
that is
\[
J (p_1,q_1,\dots, p_n,q_n) = (-q_1,p_1,\dots,-q_n,p_n),
\]
and with the symplectic form given by minus the imaginary part of the corresponding Hermitian product, that is 
\[
\Omega = \sum_{j=1}^n dp_j \wedge dq_j.
\]
H.\ Federer refers to the next result as to the Wirtinger inequality, see \cite[section 1.8.1]{fed69}.

\begin{lem}
\label{comass}
Let $1\leq k\leq n$ and let $\Omega^k$ be the $k$-times wedge product of $\Omega$ by itself. Then 
\[
\bigl| \Omega^k [u_1,\dots,u_{2k} ] \bigr| \leq k! \, |u_1 \wedge \cdots \wedge u_{2k}|, \quad \forall u_1,\dots,u_{2k} \in \R^{2n},
\]
and, in the non-trivial case of linearly independent vectors $u_j$, the equality holds if and only if the $u_j$'s span a complex subspace.
\end{lem}

We are now ready to prove the linear non-squeezing result:

\setcounter{thm}{0}

\begin{thm}[Linear non-squeezing]
Let $\Phi$ be a linear symplectic automorphism of $\R^{2n}$, and let $P: \R^{2n} \rightarrow \R^{2n}$ be the orthogonal projector onto a complex linear subspace $V\subset \R^{2n}$ of real dimension $2k$, $1\leq k\leq n$. Then
\[
\vol_{2k} \bigl( P \Phi (B_R) \bigr) \geq \omega_{2k} R^{2k},
\]
and the equality holds if and only if the linear subspace $\Phi^T V$ is complex.
\end{thm} 

\begin{proof}
By linearity, we may assume $R=1$. We consider the linear surjection
\[
A := P \Phi : \R^{2n} \rightarrow V.
\]
As before, let $\xi_1,\dots,\xi_{2k}$ be a basis of $(\ker A)^{\perp} = \ran A^T = \Phi^T V$ such that
\[
|\xi_1 \wedge \cdots \wedge \xi_{2k} | = 1.
\]
By the identity (\ref{Volume}) and Lemma \ref{comass},
\begin{equation}
\label{uno}
\begin{split}
\left( \frac{\vol_{2k} \bigl(P\Phi( B)\bigr)}{\omega_{2k}} \right)^2  = \det (A^T A|_{(\ker A)^{\perp}}) = |A^T A \xi_1 \wedge \cdots  \wedge A^T A\xi_{2k}| \\ \geq \frac{1}{k!} \bigl| \Omega^k [A^T A \xi_1,\dots,A^T A \xi_{2k} ] \bigr| = \frac{1}{k!} \bigl| \Omega^k [\Phi^T A \xi_1,\dots,\Phi^T A \xi_{2k} ] \bigr|,
\end{split} \end{equation}
and the equality holds if and only if the subspace spanned by $A^T A \xi_1, \dots, A^T A\xi_{2k}$, that is $\Phi^T V$, is complex.
Since $\Phi$ is symplectic, so is $\Phi^T$, hence
\begin{equation}
\label{due}
\Omega^k [\Phi^T A \xi_1,\dots,\Phi^T A \xi_{2k} ] = \Omega^k [A \xi_1,\dots,A \xi_{2k} ].
\end{equation}
Since the restriction of $\Omega^k$ to the complex subspace $V$ is $k!$-times the standard volume form, we have
\begin{equation}
\label{tre}
\frac{1}{k!} \bigl| \Omega^k [A \xi_1,\dots,A \xi_{2k} ] \bigr| = |A \xi_1 \wedge \cdots \wedge A\xi_{2k}| = \frac{\vol_{2k} \bigl(A(B)\bigr)}{\omega_{2k} } = \frac{\vol_{2k} \bigl(P \Phi (B)\bigr)}{\omega_{2k}},
\end{equation}
where we have used again (\ref{Volume}). By (\ref{uno}), (\ref{due}), and (\ref{tre}) we conclude that
\begin{equation}
\label{linnonsque}
\vol_{2k} (P \Phi (B)) \geq  \omega_{2k},
\end{equation}
and that the equality holds if and only if the linear subspace $\Phi^T V$ is complex.
\end{proof} 

\begin{rem}
\label{sections}
Unlike orthogonal projections onto complex subspaces, complex sections of the image of the unit ball $B$ by a linear symplectic automorphism $\Phi$ have small $2k$-volume: if $V\subset \R^{2n}$ is a complex linear subspace of real dimension $2k$, then
\[
\mathrm{vol}_{2k} \bigl( V \cap \Phi (B) \bigr) \leq \omega_{2k},
\]
and the equality holds if and only if the subspace $\Phi^{-1} V$ is complex. Indeed, since the restriction of $\Omega^k$ to $V$ is $k!$-times the Euclidean volume form, we have
\begin{eqnarray*}
\mathrm{vol}_{2k} \bigl( V \cap \Phi (B) \bigr) &=& \frac{1}{k!} \int_{V\cap \Phi (B)} \Omega^k = \frac{1}{k!} \int_{B \cap \Phi^{-1} V} \Phi^* \Omega^k = \frac{1}{k!} \int_{B \cap \Phi^{-1} V}  \Omega^k \\ &\leq& \mathrm{vol}_{2k} \bigl( B \cap \Phi^{-1} V  \bigr) = \omega_{2k},
\end{eqnarray*}
where the inequality is an equality if and only if the restriction of $\Omega^k$ to $\Phi^{-1} V$ coincides with the Euclidean volume, that is if and only if $\Phi^{-1} V$ is complex.
\end{rem}

\section{Nonlinear squeezing}
\label{2sec}

In this section we consider balls in spaces of different dimension, hence we adopt the notation $B_R^n$ for the open ball of radius $R$ about the origin in $\R^n$ and we omit the index $R$ when the radius is 1, $B^n := B^n_1$.

We denote by $\Sigma$ the punctured torus $\T^2 \setminus \{\mbox{pt}\}$ equipped with a symplectic form of area one. The following lemma is due to L.\ Guth \cite[Section 2, Main Lemma]{gut08}:

\begin{lem}
\label{guth1}
For every $R>0$ there exists a smooth symplectic embedding of $B^4_R$ into $\Sigma \times \R^2$.
\end{lem}

The next lemma is a simple modification of Lemma 3.1 in \cite{gut08} (where an embedding with extra properties is constructed for $R<1/10$):

\begin{lem}
\label{guth2}
For every $R>0$ there exists a smooth symplectic embedding of $B^2_R \times \Sigma$ into $\R^4$.
\end{lem}

\begin{proof}
Choose a positive number $\epsilon< R/3$ and set
\[
S := [-2R,2R] \times ]-\epsilon,\epsilon[, \quad S' := ]-\epsilon,\epsilon[ \times [-2R,2R].
\]
If $\epsilon$ is small enough (precisely, if $8R\epsilon<1$),  we can find a smooth symplectic immersion $\psi: \Sigma\rightarrow \R^2$ such that:
\begin{enumerate}
\item $\psi(\Sigma)\cap [-2R,2R]^2 = S \cup S'$;
\item $\psi^{-1}(S\cap S')$ consists of two disjoint open disks $D,D'\subset \Sigma$;
\item the restrictions $\psi|_{\Sigma \setminus \overline{D}}$ and $\psi|_{\Sigma \setminus \overline{D'}}$ are injective.
\end{enumerate}
Such a symplectic immersion is easily found by starting from a smooth immersion (see \cite[Figure 3]{gut08}) and by making it area-preserving by the already mentioned theorem of Moser. 

Let $\chi$ be a smooth real valued function on $\R$ with support in $[-2R+\epsilon,2R-\epsilon]$, such that $\chi=2R$ on $[-\epsilon,\epsilon]$ and $\|\chi'\|_{\infty} \leq 3/2$ (such a function exists because $\epsilon<R/3$). The map
\[
\phi: \R^4 \rightarrow \R^4, \quad \phi(p_1,q_1,p_2,q_2) := \bigl(p_1,q_1+ \chi(p_2),
p_2,q_2 + \chi'(p_2) p_1\bigr),
\]
is a symplectic diffeomorphism, being the time-one map of the Hamiltonian flow generated by the Hamiltonian $H(p_1,q_1,p_2,q_2) := \chi(p_2)p_1$. 

We claim that the map
\[
\varphi: B^2_R \times \Sigma \rightarrow \R^4, \quad \varphi = \left\{ \begin{array}{ll} \mathrm{id} \times \psi & \mbox{on } B^2_R \times (\Sigma\setminus \psi^{-1}(S)) , \\ \phi\circ (\mathrm{id} \times \psi) & \mbox{on } B^2_R \times \psi^{-1}(S) , \end{array} \right.
\]
is a symplectic embedding. 

The map $\mathrm{id}\times \psi$ maps a neighborhood of the boundary of $B^2_R \times \psi^{-1}(S)$ in $B^2_R \times \Sigma$ into a small neighborhood of $B^2_R \times \{\pm 2R\} \times ]-\epsilon,\epsilon[ $, on which $\phi=\mathrm{id}$. This proves that $\varphi$ is a smooth symplectic immersion. 

There remains to check that $\varphi$ is an embedding. By the properties of $\psi$, $\varphi$ is an embedding on a neighborhood of $B^2_R \times \overline{\Sigma \setminus \psi^{-1}(S)}$ and on a neighborhood of $B^2_R \times \psi^{-1}(S)$. Therefore, it is enough to prove that $\varphi$ maps the sets $B^2_R \times (\Sigma \setminus \psi^{-1}(S))$ and $B^2_R\times \psi^{-1}(S)$ into disjoint sets. Let $z\in \psi^{-1}(S)$, set $(p_2,q_2):= \psi(z)\in S$ and let $(p_1,q_1)\in B^2_R$. Then
\begin{equation}
\label{imago}
\varphi(p_1,q_1,z) = \bigl(p_1,q_1+ \chi(p_2),
p_2,q_2 + \chi'(p_2) p_1\bigr).
\end{equation}
Since
\[
|q_2+\chi'(p_2) p_1| \leq \epsilon + \|\chi'\|_{\infty} R < \epsilon + \frac{3}{2} R < 2R,
\]
the point $\varphi(p_1,q_1,z)$ belongs to $\R^2 \times [-2R,2R]^2$. The intersection of the latter set with 
\[
\varphi \bigl( B^2_R \times (\Sigma \setminus \psi^{-1}(S)) \bigr) = 
B^2_R \times \psi ((\Sigma\setminus \psi^{-1}(S))
\]
is $B^2_R\times S'$, so we must show that $\varphi(p_1,q_1,z)$ does not belong to $B^2_R\times S'$. If $|p_2|\geq \epsilon$, (\ref{imago}) shows that the last two-dimensional component of $\varphi(p_1,q_1,z)$ does not belong to $S'$. If $|p_2|< \epsilon$, the second component of $\varphi(p_1,q_1,z)$ is 
\[
q_1 + 2R \geq R,
\]
hence the first two-dimensional component of $\varphi(p_1,q_1,z)$ does not belong to $B^2_R$. This concludes the proof of Lemma \ref{guth2}.
\end{proof} 

We are now ready to prove the nonlinear squeezing result:

\begin{thm}[Nonlinear squeezing]
Let $P:\R^{2n}\rightarrow \R^{2n}$ be the orthogonal projection onto a complex linear subspace of $\R^{2n}$ of real dimension $2k$, with $2\leq k \leq n-1$. For every $\epsilon>0$ there exists a smooth symplectic embedding $\phi: B^{2n} \rightarrow \R^{2n}$ such that
\[
\mathrm{vol}_{2k} \bigl( P \phi(B^{2n}) \bigr) < \epsilon.
\]
\end{thm}

\begin{proof}
Up to the composition by a unitary automorphism of $(\R^{2n},J)$, we may assume that $V$ is the linear subspace corresponding to the coordinates $p_1,q_1,\dots,p_k,q_k$. Moreover, it is enough to consider the case $n=3$ and $k=2$, from which the general case follows by taking the product by the identity mapping. Because of these simplifications, $P:\R^6 \rightarrow \R^6$ is the standard projection on the subspace given by the first four coordinates $p_1,q_1,p_2,q_2$.

Let $R$ be a positive number. By Lemmata \ref{guth1} and \ref{guth2}, there are symplectic embeddings
\[
\varphi: B^4_R \rightarrow \Sigma \times \R^2, \quad \psi: B^2_R \times \Sigma \rightarrow \R^4.
\]
Consider the symplectic embedding $\tilde{\phi}:B^6_R \rightarrow \R^6$ which is defined as the composition
\[
B^6_R \hookrightarrow B^2_R \times B^4_R \stackrel{\mathrm{id}\times \varphi}{\longrightarrow} B^2_R \times \Sigma \times \R^2 \stackrel{\psi \times \mathrm{id}}{\longrightarrow} \R^4 \times \R^2 = \R^6.
\]
Then
\begin{equation}
\label{volume}
\mathrm{vol}_4\bigl(P \tilde{\phi} (B^6_R) \bigr) \leq \mathrm{vol}_4 \bigl(\psi(B^2_R \times \Sigma)\bigr) = \mathrm{vol}_4 (B^2_R \times \Sigma) = \pi R^2.
\end{equation}

The required symplectic embedding $\phi:B^6_1 \rightarrow \R^6$ is obtained by rescaling: Indeed, the embedding $\phi(z) := \tilde{\phi}(Rz)/R$ is symplectic and by (\ref{volume}), the quantity
\[
\mathrm{vol}_4 \bigl(P \phi(B^6_1) \bigr) = \mathrm{vol}_4 \Bigl( \frac{1}{R} P \tilde{\phi}(B^6_R) \Bigr) = \frac{1}{R^4} \mathrm{vol}_4\bigl(P \tilde{\phi} (B^6_R) \bigr) \leq \frac{\pi}{R^2}
\]
is smaller than $\epsilon$, if $R$ is large enough.
\end{proof} 

\section{Preliminary remarks on the local question}
\label{3sec}

Let $\phi:\R^{2n} \rightarrow \R^{2n}$ be a symplectic diffeomorphism and let $P:\R^{2n} \rightarrow \R^{2n}$ be the orthogonal projector onto a complex linear subspace $V$ of real dimension $2k$, with $2\leq k \leq n-1$. Denote by $\psi$ the composition $P \phi$. In its former formulation, the local question proposed in the introduction is whether the inequality
\begin{equation}
\label{middle3}
\vol_{2k} \bigl(\psi (B_R) \bigr) \geq \omega_{2k} R^{2k},
\end{equation}
holds for $R>0$ small enough. 

The linear non-squeezing Theorem 1, together with the identities (\ref{Volume}) and (\ref{maxi}), implies that
\begin{equation}
\label{lincon}
J_{2k} \psi(x) \geq 1, \qquad \forall x\in \R^{2n},
\end{equation}
where $J_{2k}\psi (x)$ is the $2k$-Jacobian of the map $\psi$ at $x$, that is the number
\[
J_{2k} \psi(x) = \max_{W\in \mathrm{Gr}_{2k}(\R^{2n})} |\det D\psi(x) |_W|,
\] 
$\mathrm{Gr}_{2k}(\R^{2n})$ being the Grassmannian of $2k$-dimensional linear subspaces of $\R^{2n}$.

Our first remark is that (\ref{middle3}) does not follow simply from the inequality (\ref{lincon}). In fact, if $m>h>1$, there are smooth maps $\varphi:\R^m \rightarrow \R^h$  whose $h$-Jacobian is everywhere at least 1 but for which
\begin{equation}
\label{strizza}
\mathrm{vol}_h \bigl( \varphi(B_R^m) \bigr) < \omega_h R^h,
\end{equation}
for every small $R>0$. Examples with $m=3$ and $h=2$ can be found among the maps of the form
\begin{equation}
\label{example}
\varphi: \C \times \R \cong \R^3 \rightarrow \C \cong \R^2, \quad
\varphi(z,t) = \rho(|z|) e^{it} z,
\end{equation}
where $\rho:[0,+\infty[ \rightarrow \R$ is a smooth positive function such that
\[
\rho(0)=1, \quad \rho'(0)=0, \quad \rho(r)<1 \quad \forall r>0.
\]
Indeed, for small $R$ such a $\varphi$ maps the cylinder $B^2_R \times \R$ -- and a fortiori the ball $B^3_R$ -- into the disk of radius $\rho(R)R$, which is smaller than $R$ for $R>0$. The $2$-Jacobian  of $\varphi$ is easily computed to be 
\[
J_2 \varphi (z,t)  = \rho(|z|) \bigl( \rho(|z|) + |z| \rho'(|z|) \bigr) \sqrt{1+|z|^2}, 
\]
from which we find
\[
J_2 \varphi (z,t) = 1 + \left( \frac{1}{2} + 2\rho''(0) \right) |z|^2 + O(|z|^3) \qquad \mbox{for } z\rightarrow 0.
\]
Therefore, if $\rho''(0)>-1/4$ then
$J_2 \varphi (z,t) \geq 1$ for $|z|$ small enough. Examples for arbitrary $m>h>1$ are obtained from this map by taking the product with an orthogonal projection.

\medskip

The existence of smooth maps $\varphi:\R^m \rightarrow \R^h$ with $m>h$ whose $h$-Jacobian is everywhere at least 1 but for which the inequality (\ref{strizza}) holds for every small $R>0$, such as the one defined above, is related to a non-integrability issue. Indeed, the set 
\begin{equation}
\label{mvd}
\mathscr{W}(x) := \set{ W\in \mathrm{Gr}_h (\R^m) }{ |\det D\varphi(x) |_W| \geq 1}
\end{equation}
is non-empty for every $x\in \R^m$ because $J_h \varphi(x)\geq 1$ and defines what could be called a ``multi-valued $h$-dimensional distribution'' on $\R^m$. If $\varphi$ satisfies (\ref{strizza}) for every small $R>0$, such a multi-valued distribution cannot be integrable in a neighborhood of zero, as is implied by the following result (which also clarifies the meaning of integrability for multi-valued distributions):

\begin{lem}
\label{seintespa}
Let $m\geq h\geq 1$ and let $\varphi:\R^m \rightarrow \R^h$ be a smooth map such that $J_h(\varphi(x))\geq 1$ for every $x\in \R^m$. Assume that the multi-valued distribution $\mathscr{W}$, which is defined in (\ref{mvd}), is integrable, in the sense that $\R^m$ admits a smooth $h$-dimensional foliation $\mathscr{F}$ such that for every $x\in \R^m$ the tangent space at $x$ of the leaf through $x$ belongs to $\mathscr{W}(x)$.
Then
\[
\mathrm{vol}_h \bigl( \varphi (B_R^m) \bigr) \geq \omega_h R^h,
\]
for every $R>0$ small enough.
\end{lem}

The proof of this lemma is given at the end of this section. Notice that when $h=1$ or $h=m$, $\mathscr{W}$ is always integrable: the case $h=m$ is obvious, because $\mathscr{W}(x)=\R^m$ for every $x$, and the case $h=1$ follows from the fact that, by (\ref{maxi}), $\mathscr{W}$ admits the smooth selection
\[
W(x) = D\varphi (x)^T \R^h,
\]
which, when $h=1$, is a genuine one-dimensional distribution, hence integrable.
So the above lemma explains our previous claim that maps $\varphi:\R^m\rightarrow \R^h$ which satisfy $J_h \varphi\geq 1$ and (\ref{strizza}) for every small $R>0$ exist only when $m>h>1$.

\medskip

Now let $\phi:\R^{2n} \rightarrow \R^{2n}$ be a symplectic diffeomorphism and let $P$ be the orthogonal projector onto a complex linear subspace $V\subset \R^{2n}$. 
The above lemma implies a positive answer to the local non-squeezing question when the multi-valued distribution $\mathscr{W}$ which is associated to the map $P \phi$ is integrable. In the ``rigid case'', in which $\mathscr{W}(x)$ consists of a single vector space for every $x\in \R^{2n}$, $\mathscr{W}$ is certainly integrable. In fact, in this case formula (\ref{maxi}) and Theorem 1 imply that $\mathscr{W}(x)$ consists of the space
\[
W(x) = D\phi(x)^T V,
\]
which must be complex for every $x\in \R^{2n}$.  Therefore, the distribution
\[
W(x) = J W(x) = JD\phi(x)^T V = D\phi(x)^{-1}J V = D\phi(x)^{-1} V
\]
is tangent to the foliation given by the image by $\phi^{-1}$ of the linear foliation given by $2k$-dimensional subspaces parallel to $V$. This fact has the following consequence:

\begin{prop*}[Generic local non-squeezing]
Let $\phi:\R^{2n}\rightarrow \R^{2n}$ be a symplectic diffeomorphism and let $P$ be the orthogonal projection onto a complex subspace $V\subset \R^{2n}$ of real dimension $2k$, with $1\leq k \leq n$. Then there is an open and dense subset $U\subset \R^{2n}$ with the following property: for every $x\in U$ there exists $R_0=R_0(x)>0$ such that
\begin{equation}
\label{sss}
\vol_{2k} \bigl( P\phi (B_R(x)) \bigr) \geq \omega_{2k} R^{2k},
\end{equation}
for every $R\leq R_0$. 
\end{prop*}

Indeed, we can consider the closed set $Y\subset \R^{2n}$ consisting of all points $x$ for which 
\[
\mathrm{vol}_{2k} \bigl( P D\phi(x) (B) \bigr) = \omega_{2k},
\]
or, equivalently,
\[
\mathscr{W}(x) = \{W(x)\} = \{ D\phi(x)^T V\}.
\]
On the complement of $Y$ we have
\[
\mathrm{vol}_{2k} \bigl( P D\phi(x) (B) \bigr) > \omega_{2k},
\]
and (\ref{sss}) holds for every small $R$ just by continuity. On the interior part of $Y$, the distribution $W$ is integrable, as we have seen, and (\ref{sss}) holds for every  small $R$ because of Lemma \ref{seintespa}. Then
\[
U:= \mathrm{Int}(Y) \cup Y^c = \R^{2n} \setminus \partial Y
\]
is the required open and dense set.
\medskip

Although the analysis of the above and of several other examples suggests that the multi-valued distribution associated to $P\phi$ should always be integrable, we do not have a proof of this fact. Therefore, in the following sections we use a different and more direct strategy for studying the local question. 

\medskip

We conclude this section by proving Lemma \ref{seintespa}. The proof is based on the following result:

\begin{lem}
\label{foliation}
Let $\mathscr{F}$ be a foliation of $B^m$ consisting of $h$-dimensional disks. If $\mathscr{F}$ is $C^2$-close enough to the foliation by parallel affine subspaces, then it has a leaf whose $h$-volume is at least $\omega_h$.
\end{lem}

\begin{proof}
By assumption, the leaves of $\mathscr{F}$ are graphs of $\R^{m-h}$-valued maps which are defined on convex domains in $\R^h$. Let $F$ be a leaf in $\mathscr{F}$. Since $F$  is $C^2$-close to a flat $h$-disk, the Plateau problem for minimal $h$-surfaces with boundary $F\cap \partial B$ can be solved in the class of graphs of smooth maps and produces a unique $h$-disk $\tilde{F}$ with minimal $h$-volume among the $h$-disks with the same boundary (see \cite[Theorem 1.1]{wan03}).  Moreover, the estimates of \cite{wan03} allow to prove that $\tilde{F}$ depends continuously on $F\in \mathscr{F}$. In particular, there is a $F\in \mathscr{F}$ such that $0$ belongs to $\tilde{F}$. Then the conclusion follows from the monotonicity formula for minimal submanifolds (see e.g.\ \cite[Corollary 1.13]{cm11}), which implies the second inequality in
\[
\mathrm{vol}_h (F)  \geq
\mathrm{vol}_h (\tilde{F}) \geq \omega_h.
\]
\end{proof}

\begin{rem}
The conclusion of Lemma \ref{foliation} should hold also if $\mathscr{F}$ is not assumed to be $C^2$-close to the affine foliation, and actually also for more general objects than smooth foliations. In the case of the sphere, results of this kind have been proved by F.\ J.\ Almgren in the unpublished manuscript \cite{alm} and by M.\ Gromov in \cite{gro03} (see also \cite{mem11} for a detailed proof of Gromov's theorem). The case of the ball is discussed in \cite[Appendix 1.F]{gro83} and in \cite[Section 2]{gut09}.
\end{rem}

\begin{proof}[Proof of Lemma \ref{seintespa}]
If $R$ is small enough, the homothety $x\mapsto x/R$ maps the foliation $\mathscr{F}|_{B^m_R}$ into a foliation of $B^m$ which satisfies the assumptions of Lemma \ref{foliation}. We deduce that $\mathscr{F}|_{B^m_R}$ has a leaf $F$ such that
\[
\mathrm{vol}_h (F) \geq \omega_h R^h.
\]
Since $D\varphi (0)|_{T_0 F}$ is an isomorphism,
up to the choice of a smaller $R$ we can also assume that the restriction of $\varphi$ to $F$ is injective. Then the fact that 
\[
\bigl|\det (D\varphi (x)|_{T_x F}) \bigr|\geq 1, \qquad \forall x\in F,
\]
and the area formula imply that
\[
\vol_h \bigl(\varphi (B^m_R) \bigr) \geq \vol_h \bigl(\varphi (F) \bigr) \geq \vol_h (F) \geq 
\omega_h R^h,
\]
as we wished to prove.
\end{proof}

\section{Preparatory results for Theorem \ref{expa}}
\label{4sec}

Let $1\leq k \leq n$ be fixed once and for all.
In the proof of Theorem \ref{expa} we shall make use of some auxiliary differential forms on $\R^{2n}$, which we now introduce. The symplectic form $\Omega$ is decomposed as 
\[
\Omega = \Omega_k + \widehat{\Omega},
\]
where
\[
\Omega_k := \sum_{j=1}^k dp_j \wedge dq_j, \qquad \widehat{\Omega} := \sum_{j=k+1}^n dp_j \wedge dq_j.
\]
The next object is the $(2k-1)$-form
\[
\alpha:= p_1 dq_1 \wedge dp_2 \wedge dq_2 \wedge \dots \wedge dp_k \wedge 
dq_k,
\]
whose differential is
\[
d\alpha = dp_1 \wedge dq_1 \wedge  \dots \wedge dp_k \wedge 
dq_k = \frac{1}{k!} \Omega_k^k.
\]
Moreover
\begin{equation}
\label{alfabeta}
\Omega^k = (\Omega_k + \widehat{\Omega})^k = \sum_{j=0}^k \binom{k}{j} \Omega_k^{k-j} \wedge \widehat{\Omega}^j = \Omega_k^k + \sum_{j=1}^k \binom{k}{j} \Omega_k^{k-j} \wedge \widehat{\Omega}^j = k! (d\alpha + \beta),
\end{equation}
where $\beta$ is the $2k$-form
\begin{equation}
\label{forbeta}
\beta:= \frac{1}{k!} \sum_{j=1}^k \binom{k}{j} \Omega_k^{k-j} \wedge \widehat{\Omega}^j = \sum_{\substack{1\leq i_1 < \dots < i_k \leq n \\ i_k > k}} dp_{i_1} \wedge dq_{i_1} \wedge \dots \wedge dp_{i_k} \wedge dq_{i_k}.
\end{equation}
In the following lemma we prove two formulas which will be useful in the next sections.

\renewcommand{\theenumi}{\roman{enumi}}
\renewcommand{\labelenumi}{(\theenumi)}

\begin{lem}
\label{formule}
Let $P: \R^{2n} \rightarrow \R^{2n}$ be the orthogonal projector onto the $2k$-dimensional subspace which corresponds to the coordinates $p_1,q_1,\dots,p_k,q_k$. If $X$ and $Y$ are smooth vector fields on $\R^{2n}$, then there holds:
\begin{enumerate}
\item $P^* ( \imath_X d \imath_Y \beta) = \frac{1}{(k-1)!} P^* \bigl( \Omega_k^{k-1} \wedge ( \imath_X d \imath_Y \widehat{\Omega} ) \bigr)$.
\end{enumerate}
If $X_H$ is the Hamiltonian vector field associated to the Hamiltonian $H\in C^{\infty}(\R^{2n})$, then there holds:
\begin{enumerate}
\setcounter{enumi}{1}
\item $P^* ( \imath_{X_H} d\alpha ) = - \frac{1}{(k-1)!} d P^* (H\,\Omega^{k-1})$.
\end{enumerate}
\end{lem}

\begin{proof}
We clearly have
\begin{equation}
\label{parte}
P^* dp_j = \left\{ \begin{array}{ll} dp_j & \mbox{if } j \leq k, \\ 0 &\mbox{if } j>k , \end{array} \right. \qquad P^* dq_j = \left\{ \begin{array}{ll} dq_j & \mbox{if } j \leq k, \\ 0 &\mbox{if } j>k. \end{array} \right.
\end{equation}
It follows that
\begin{eqnarray}
\label{vanish1}
P^* \widehat{\Omega} &=& 0, \\
\label{vanish2}
P^* (\imath_X \widehat{\Omega}) &=& 0,
\end{eqnarray}
for every vector field $X$.

\medskip

\noindent (i) If $k<i<j\leq n$, the three-form
\[
\imath_{X} d \imath_{Y} (dp_i \wedge dq_i \wedge dp_j \wedge dq_j)
\]
is a sum of elementary three-forms, each of which contains at least a term $dp_h$ or $dq_h$, where $h$ is either $i$ or $j$. By (\ref{parte}), it follows that
\[
P^*\bigl( \imath_{X} d \imath_{Y} (dp_i \wedge dq_i \wedge dp_j \wedge dq_j) \bigr) = 0.
\]
Then, by the last expression for $\beta$ in (\ref{forbeta}), we find
\[
P^* ( \imath_{X} d \imath_{Y} \beta ) = P^* ( \imath_{X} d \imath_{Y} \hat{\beta} ),
\]
where
\[
\hat{\beta} := \sum_{1\leq i_1 < \dots < i_{k-1} \leq k < i_k \leq n} dp_{i_1} \wedge dq_{i_1} \wedge \dots \wedge dp_{i_k} \wedge dq_{i_k}.
\]
Notice that
\[
\hat{\beta} = \left( \sum_{1\leq i_1 < \dots < i_{k-1} \leq k} dp_{i_1} \wedge dq_{i_1} \wedge \dots \wedge dp_{i_{k-1}} \wedge dq_{i_{k-1}} \right) \wedge \widehat{\Omega} = \frac{1}{(k-1)!} \Omega_k^{k-1} \wedge \widehat{\Omega},
\]
and hence
\begin{equation}
\label{finqua}
P^* ( \imath_{X} d \imath_{Y} \beta ) = \frac{1}{(k-1)!} P^* \Bigl( \imath_{X} d \imath_{Y} \bigl(\Omega_k^{k-1} \wedge \widehat{\Omega}\bigr) \Bigr).
\end{equation}
When $k=1$, the above identity is precisely the identity (i) that we wish to prove. So we assume that $k\geq 2$ and compute
\begin{eqnarray*}
\imath_{X} d \imath_{Y} \bigl( \Omega_k^{k-1} \wedge \widehat{\Omega} \bigr) &=& \imath_{X} d \Bigl( (k-1) ( \imath_{Y} \Omega_k ) \wedge \Omega_k^{k-2} \wedge \widehat{\Omega} + \Omega_k^{k-1} \wedge ( \imath_{Y} \widehat{\Omega} ) \Bigr) \\
&=& \imath_{X} \Bigl( (k-1) (d \imath_{Y} \Omega_k ) \wedge \Omega_k^{k-2} \wedge \widehat{\Omega} + \Omega_k^{k-1} \wedge ( d\imath_{Y} \widehat{\Omega}) \Bigr).
\end{eqnarray*} 
By (\ref{vanish1}), $P^* \widehat{\Omega}$ vanishes, and we get
\begin{eqnarray*}
P^*\Bigl( \imath_{X} d \imath_{Y} \bigl( \Omega_k^{k-1} \wedge \widehat{\Omega} \bigr) \Bigr) &=& P^* \Bigl( \imath_{X} \bigl( \Omega_k^{k-1} \wedge ( d\imath_{Y} \widehat{\Omega} \bigr) \Bigr) \\ &=& P^* \Bigl( (k-1) ( \imath_{X} \Omega_k) \wedge \Omega_k^{k-2} \wedge (d\imath_{Y} \widehat{\Omega}) + \Omega_k^{k-1} \wedge ( \imath_{X} d \imath_{Y} \widehat{\Omega}) \Bigr).
\end{eqnarray*}
Since, by (\ref{vanish2}), the two-form $P^*(d \imath_{Y} \widehat{\Omega})=d P^* (\imath_{Y} \widehat{\Omega})$ vanishes, we conclude that
\[
P^*\Bigl( \imath_{X} d \imath_{Y} \bigl( \Omega_k^{k-1} \wedge \widehat{\Omega} \bigr) \Bigr) = P^* \Bigl(  \Omega_k^{k-1} \wedge ( \imath_{X} d \imath_{Y} \widehat{\Omega}) \Bigr).
\]
The thesis in the case $k\geq 2$ follows from the above identity and (\ref{finqua}).

\medskip

\noindent (ii) Since all the summands in the definition of $\beta$ contain the term $\widehat{\Omega}$, (\ref{vanish2}) implies that
\begin{equation}
\label{null}
P^* ( \imath_{X_H} \beta ) = 0.
\end{equation}
Moreover,
\[
\imath_{X_H} \Omega^k = k (\imath_{X_H} \Omega) \wedge \Omega^{k-1} = - k \, dH \wedge \Omega^{k-1} = - k \, d \bigl( H \, \Omega^{k-1} \bigr).
\]
Therefore, by (\ref{alfabeta}),
\[
\imath_{X_H} d\alpha  = \frac{1}{k!} \imath_{X_H} \Omega^k - \imath_{X_H} \beta = - \frac{1}{(k-1)!} d(H\, \Omega^{k-1})-  \imath_{X_H} \beta.
\]
By applying $P^*$, taking (\ref{null}) into account, the formula (ii) follows.
\end{proof}

\section{Strategy of the proof of Theorem \ref{expa}}
\label{5sec}

Let $\{\phi_t\}_{t\in [0,1]}$ be a smooth one-parameter family of symplectic diffeomorphisms of $\R^{2n}$ such that $\phi_0=\Phi$ is linear, and let $\{H_t\}_{t\in [0,1]}$ be its generating path of Hamiltonians, that is
\[
\frac{d}{dt} \phi_t(x) = X_{H_t} (\phi_t(x)), \qquad \forall (t,x)\in [0,1]\times \R^{2n},
\]
where $X_{H_t}$ is defined by the identity
\[
\imath_{X_{H_t}} \Omega = - dH_t.
\]
Let $P$ be the orthogonal projector onto the subspace $\R^{2k}$ which is given by the coordinates $p_1,q_1,\dots,p_k,q_k$. This subspace is complex with respect to the complex structure $J$, hence $P$ commutes with $J$.
We assume that also the subspace $\Phi^T \R^{2k}$ is complex.

Let $U$ be an orthogonal and symplectic automorphism of $\R^{2n}$ - that is unitary with respect to the complex structure $J$ - which maps the complex subspace $\R^{2k}$ onto the complex subspace $\Phi^T \R^{2k}$, and set
\[
\Psi := \Phi U.
\]
Then 
\begin{equation}
\label{preserva}
\Psi^T \R^{2k} = (\Phi U)^T \R^{2k} = U^T \Phi^T \R^{2k} = U^{-1} \Phi^T \R^{2k} = \R^{2k},
\end{equation}
so $\Psi^T$ preserves the subspace $\R^{2k}$. Since $\Psi^T$ is symplectic, it preserves also the symplectic orthogonal of $\R^{2k}$, which is also its Euclidean orthogonal. It follows that $\Psi^T$ commutes with the orthogonal projector $P$, and so does $\Psi$:
\[
[\Psi,P] = 0.
\] 
Being obtained from $\phi_t$ by right-composition, the path of symplectic diffeomorphisms
\[
\psi_t(x) := \phi_t ( U x ), 
\]
satisfies
\begin{equation}
\label{hamsist}
\begin{split}
\frac{d}{dt} \psi_t (x) &= X_{H_t} (\psi_t(x)), \\
\psi_0 &= \Psi.
\end{split}
\end{equation}
Since $U$ maps the open unit ball $B\subset \R^{2n}$ onto itself, we have
\[
\psi_t (B) = \phi_t(B).  
\]
Theorem \ref{expa} calls for finding the second order expansion of the function
\[
f(t) :=  \mathrm{vol}_{2k} \bigl( P \phi_t (B) \bigr) =  \mathrm{vol}_{2k} \bigl( P \psi_t (B) \bigr)
\]
at $t=0$. By (\ref{inpiu}) and (\ref{preserva}), we have
\[
P \psi_0(B) = P \Psi (B) = P \Psi ( B\cap \ran (P\Psi)^T ) = P \Psi ( B \cap \Psi^T \R^{2k} ) = P \Psi (B\cap \R^{2k}) = P \Psi(B^{2k}).
\]
Therefore, 
\[
\partial P \psi_0 (B) = P \psi_0 (S^{2k-1}),
\]
where $S^{2k-1}=\partial B^{2k}$ denotes the unit sphere in $\R^{2k}$. The implicit mapping theorem implies that if $\bar{t}\in ]0,1]$ is small enough, there is a smooth one-parameter family of embeddings
\[
x_t : S^{2k-1} \hookrightarrow \partial B, \qquad t\in [0,\bar{t}],
\]
such that
\begin{equation}
\label{bordo}
 \partial P \psi_t (B)= P  \psi_t (x_t(S^{2k-1})), \qquad \forall t\in [0,\bar{t}\,],
\end{equation}
and $x_0$ is the identity mapping on $S^{2k-1}$. Actually, we may also impose the normalization condition
\begin{equation}
\label{norm}
P x_t(\theta) \in \R^+ x_0(\theta), \qquad  \forall \theta\in S^{2k-1},
\end{equation}
which uniquely determines the embeddings $x_t$. 

Given a smooth map $x:S^{2k-1} \rightarrow \R^{2n}$, we set
\[
V_t(x) := \int_{S^{2k-1}} (\psi_t \circ x)^* \alpha,
\]
the $(2k-1)$-form $\alpha$ being defined in Section \ref{4sec}.
If $x$ is the boundary of the map $\tilde{x}:B^{2k} \rightarrow \R^{2n}$, Stokes theorem and the identity $P^* \alpha = \alpha$ (see (\ref{parte})) imply that
\begin{eqnarray*}
V_t (x) &=& \int_{S^{2k-1}} (\psi_t \circ x)^* \alpha = \int_{B^{2k}} (\psi_t \circ \tilde{x})^* d\alpha =  \int_{B^{2k}} (\psi_t \circ \tilde{x})^* P^* d\alpha \\ &=& \int_{B^{2k}} (P\circ \psi_t \circ \tilde{x})^* d\alpha  = \int_{B^{2k}} (P\circ \psi_t \circ \tilde{x})^* dp_1 \wedge dq_1 \wedge \dots \wedge dp_k \wedge dq_k.
\end{eqnarray*}
When $P\circ \psi_t \circ \tilde{x}$ is one-to-one and orientation preserving, then
\[
V_t(x) = \int_{P \circ \psi_t \circ \tilde{x} (B^{2k})}  dp_1 \wedge dq_1 \wedge \dots \wedge dp_k \wedge dq_k = \mathrm{vol}_{2k} \bigl( P\circ \psi_t \circ \tilde{x} (B^{2k})\bigr).
\]
Up to the choice of a smaller $\bar{t}\in ]0,1]$, the above formula holds in particular when $x=x_t$, $t\in [0,\bar{t}]$, for a suitable choice of $\tilde{x}=\tilde{x}_t:B^{2k} \rightarrow B$, because $x_t$ is $C^1$-close to the identity mapping on $S^{2k-1}$ and $\psi_t$ is $C^1$-close to the linear map $\Psi$, which preserves the splitting determined by $P$.  Hence
\begin{equation}
\label{dausare}
f(t) = V_t(x_t) = \int_{S^{2k-1}} (\psi_t \circ x_t)^* \alpha, \qquad \forall t\in [0,\bar{t}].
\end{equation}
Moreover, by (\ref{bordo}) the map $x_t$ is a local maximum of $V_t$ on the space of smooth maps $x:S^{2k-1} \rightarrow \partial B$. Our strategy for proving Theorem \ref{expa} is to use the formula (\ref{dausare}) in order to determine the second order expansion of $f(t)$ at $t=0$.

\section{Computations} 
\label{6sec}

The first aim of this section is to get some information on the family of embeddings $x_t : S^{2k-1} \rightarrow \partial B$, $t\in [0,\bar{t}]$, which is defined by (\ref{bordo}) and (\ref{norm}). 

\begin{lem}
\label{super}
\begin{enumerate} 
\item For every $t\in [0,\bar{t}]$, the map $x_t$ satisfies the functional equation
\[
(I-P)\,  D\psi_t \bigl(x_t(\theta)\bigr) \bigl[ J x_t (\theta)\bigr] = 0, \qquad \forall \theta\in S^{2k-1}.
\]
\item There holds
\[
\frac{d x_t}{d t} \Big|_{t=0} (\theta) = J \Psi^{-1} (I-P) D X_{H_0}(\Psi x_0(\theta)) \Psi J x_0(\theta), \qquad \forall \theta\in S^{2k-1}.
\]
\end{enumerate}
\end{lem}

\begin{proof}
(i) Fix $t\in [0,\bar{t}]$ and $\theta\in S^{2k-1}$. The point $x:= x_t(\theta)$ is a singular point for the map
\[
\partial B \to \R^{2k}, \quad z \mapsto P \psi_t(z).
\]
Therefore, there is a non-zero vector $\eta\in \R^{2k}$ such that 
\[
P D\psi_t (x) [\xi] \cdot \eta = 0,
\]
for every $\xi$ in $T_{x} \partial B$. Since $P\eta=\eta$, we have
\[
P D\psi_t (x) [\xi] \cdot \eta = \xi \cdot D\psi_t (x)^T [\eta],
\]
from which we deduce that the vector
\[
D\psi_t (x)^T [\eta] 
\]
is orthogonal to $T_{x} \partial B$ or, equivalently, is a multiple of $x$. Since $\eta$ belongs to $\R^{2k}$, we deduce that
\[
\bigl(D\psi_t (x)^T \bigr)^{-1} [x] \in \R^{2k}.
\]
By the $J$-invariance of  $\R^{2k}$, this is equivalent to
\[
(I-P) J \bigl(D\psi_t (x)^T \bigr)^{-1} [x] = 0.
\]
Since $D\psi_t(x)$ belongs to the symplectic group, 
\[
J \bigl(D\psi_t (x)^T \bigr)^{-1} = D \psi_t(x) J,
\]
and we conclude that
\[
(I-P)\,  D \psi_t(x) [Jx] = 0,
\]
which proves (i).

\medskip

\noindent (ii)
We can write the first order expansion
\begin{equation}
\label{expansion}
x_t(\theta) = x_0 (\theta) + t y(\theta) + o(t), \qquad \mbox{for } t\rightarrow 0, \end{equation}
where the smooth map 
\[
y:S^{2k-1}\rightarrow \R^{2n}, \quad y(\theta) := \frac{d x_t}{d t} \Big|_{t=0} (\theta),
\]
is to be determined. From the condition $|x_t(\theta)|=1$ we deduce that
\begin{equation}
\label{n1}
y(\theta) \cdot x_0 (\theta) = 0, \qquad \forall \theta\in S^{2k-1},
\end{equation}
while the normalization condition (\ref{norm}) implies
\begin{equation}
\label{n2}
P y(\theta) \in \R x_0 (\theta), \qquad \forall \theta\in S^{2k-1}.
\end{equation}
By (\ref{n1}) we get
\[
0 = y \cdot x_0 = y \cdot P x_0 = Py \cdot x_0,
\]
which together with (\ref{n2}) implies that $Py=0$. Differentiating (\ref{hamsist}) with respect to $x$, we find
\[
\begin{split}
\frac{d}{dt} D \psi_t(x) &= D X_{H_t} (\psi_t(x)) D\psi_t(x), \\
D\psi_0(x) &= \Psi,
\end{split}
\]
from which we deduce that
\[
D \psi_t (x) = \Psi + t D X_{H_0}(\Psi x) \Psi + o(t), \qquad \mbox{for } t\rightarrow 0.
\]
Together with (i) and (\ref{expansion}), this implies that
\[
(I-P) \bigl( \Psi + t D X_{H_0} ( \Psi x_0 + t \Psi y + o(t) )\Psi + o(t) \bigr)  J (x_0 + ty + o(t))= 0.
\]
By expanding this expression and by collecting the terms which are linear in $t$, we find that
\begin{equation}
\label{tmp}
(I-P) \bigl( \Psi J y + DX_{H_0} (\Psi x_0) \Psi J x_0 \bigr) = 0.
\end{equation}
Since $Py=0$ and $P$ commutes with both $\Psi$ and $J$, we have $P\Psi J y=0$, and (\ref{tmp}) implies that
\[
\Psi J y = (I-P) \Psi J y = - (I-P) D X_{H_0} (\Psi x_0) \Psi J x_0,
\]
from which we conclude that
\[
y = J \Psi^{-1} (I-P) D X_{H_0} (\Psi x_0) \Psi J x_0,
\]
as claimed.
\end{proof}

We can now use the formula (\ref{dausare}) and compute the first derivative of $f$.

\begin{lem}
\label{der1}
There holds
\[
f'(t) = \int_{S^{2k-1}} (\psi_t \circ x_t)^*( \imath_{X_{H_t}} d \alpha), 
\]
and $f'(0)=0$.
\end{lem}

\begin{proof}
Since $x_t$ is a local maximum of the functional $V_t$ on the space of smooth maps $x:S^{2k-1} \rightarrow \partial B$, we have
\[
f'(t) = dV_t(x_t)\left[ \frac{\partial x_t}{\partial t} \right] + \frac{\partial V_t}{\partial t} (x_t) = \frac{\partial V_t}{\partial t} (x_t).
\]
Moreover, by using Cartan's identity
\[
L_X \alpha = \imath_X d\alpha + d \imath_{X} \alpha
\]
and Stokes' theorem, we find
\begin{eqnarray*}
\frac{\partial V_t}{\partial t} (x) &=& \frac{d}{dt}  \int_{S^{2k-1}} (\psi_t \circ x)^* \alpha =  \int_{x(S^{2k-1})} \frac{d}{dt} (\psi_t^* \alpha) \\ &=& \int_{x(S^{2k-1})} \psi_t^* ( L_{X_{H_t}} \alpha ) = \int_{x(S^{2k-1})} \psi_t^* ( \imath_{X_{H_t}} d\alpha + d \imath_{X_{H_t}} \alpha) \\ &=& \int_{S^{2k-1}} (\psi_t \circ x)^*  ( \imath_{X_{H_t}} d\alpha ) + \int_{S^{2k-1}} d (\psi_t\circ x)^* (\imath_{X_{H_t}} \alpha) = \int_{S^{2k-1}} (\psi_t \circ x)^*  ( \imath_{X_{H_t}} d\alpha ) .
\end{eqnarray*}
The formula for $f'(t)$ follows. In particular,
\[
f'(0) = \int_{S^{2k-1}} (\Psi \circ x_0)^*( \imath_{X_{H_0}} d\alpha ).
\]
Since $\Psi \circ x_0 (S^{2k-1})=\Psi S^{2k-1}$ is contained in $\R^{2k}$, Lemma \ref{formule} (ii) implies that
\begin{eqnarray*}
(\Psi \circ x_0)^*( \imath_{X_{H_0}} d\alpha ) &=& (P \circ \Psi \circ x_0)^*( \imath_{X_{H_0}} d\alpha ) = (\Psi \circ x_0)^*\bigl( P^* (\imath_{X_{H_0}} d\alpha )\bigr) \\ &=& - \frac{1}{(k-1)!} (\Psi \circ x_0)^*\bigl( d P^* (H_0 \Omega^{k-1}) \bigr) =  - \frac{1}{(k-1)!}  d (\Psi \circ x_0)^*(H_0 \Omega^{k-1}).
\end{eqnarray*}
Hence $f'(0)=0$, again by Stokes' theorem.
\end{proof}

We can now get a first formula for the second derivative of $f$.

\begin{lem}
\label{der2}
There holds
\[
f''(t) = \int_{S^{2k-1}} (\psi_t \circ x_t)^* \left( - \imath_{X_{H_t}+Y_t} d \imath_{X_{H_0}} \beta + \imath_{X_{\frac{d H_t}{d t}}} d\alpha \right),
\]
where $Y_t$ is a smooth vector field on $\R^{2n}$ such that
\[
Y_t\bigl(\psi_t \circ x_t (\theta) \bigr)= D\psi_t(x_t(\theta)) \frac{dx_t}{dt} (\theta), \qquad \forall \theta\in S^{2k-1}.
\]
In particular, for $t=0$,
\[
f''(0) = - \frac{1}{(k-1)!} \int_{S^{2k-1}} x_0^* \bigl( \Omega_k^{k-1} \wedge \Psi^*(\imath_{X_{H_0}+Y} d \imath_{X_{H_0}} \widehat{\Omega} ) \bigr),
\]
where $Y$ is the vector field
\[
Y(x) := \Psi J \Psi^{-1} (I-P) D X_{H_0} (x) \Psi J \Psi^{-1} x.
\]
\end{lem}

\begin{proof}
By deriving the expression for $f'(t)$ of Lemma \ref{der1}, we find, by using Cartan's identity and Stokes' theorem as in the proof of Lemma \ref{der1},
\begin{eqnarray*}
f''(t) &=& \int_{S^{2k-1}} (\psi_t \circ x_t)^* \left( L_{Z_t} \imath_{X_{H_t}} d\alpha + \frac{d}{d t}  \imath_{X_{H_t}} d\alpha \right) \\ &=& \int_{S^{2k-1}} (\psi_t \circ x_t)^* \left( \imath_{Z_t} d \imath_{X_{H_t}} d \alpha + \frac{d}{d t}  \imath_{X_{H_t}} d\alpha \right),
\end{eqnarray*}
where $Z_t$ is a vector field on $\R^{2n}$ such that
\[ 
Z_t\bigl(\psi_t \circ x_t (\theta) \bigr) = \frac{d}{dt} \bigl( \psi_t \circ x_t (\theta) \bigr) = X_{H_t}(\psi_t \circ x_t(\theta)) + D\psi_t(x_t(\theta)) \frac{dx_t}{dt} (\theta), \qquad \forall \theta\in S^{2k-1}.
\]
Here we can write
\[
Z_t = X_{H_t} + Y_t,
\]
where $Y_t$ is a vector field on $\R^{2n}$ such that
\[
Y_t\bigl(\psi_t \circ x_t (\theta) \bigr)= D\psi_t(x_t(\theta)) \frac{dx_t}{dt} (\theta), \qquad \forall \theta\in S^{2k-1}.
\]
Since
\[
\frac{d}{d t} \imath_{X_{H_t}} d \alpha = \imath_{\frac{d}{d t} X_{H_t}} d \alpha = \imath_{X_{\frac{d H_t}{d t}}} d\alpha,
\]
we find
\begin{equation}
\label{prov}
f''(t) = \int_{S^{2k-1}} (\psi_t \circ x_t)^* \left( \imath_{X_{H_t}+Y_t} d \imath_{X_{H_t}} d\alpha + \imath_{X_{\frac{d H_t}{d t}}} d\alpha \right).
\end{equation}
From the identity (\ref{alfabeta}) we have
\[
d\alpha = \frac{1}{k!} \Omega^k - \beta,
\]
which, together with
\[
d\imath_{X_{H_t}} \Omega^k = d \bigl( k (\imath_{X_{H_t}} \Omega) \wedge \Omega^{k-1} \bigr) = d( - k dH_t \wedge \Omega^{k-1} ) = 0,
\]
implies that we can replace $d\alpha$ with $-\beta$ in the first term of (\ref{prov}), and we find the desired formula for $f''(t)$. 

For $t=0$ we have, by Lemma \ref{super} (ii),
\[
Y_0\bigl(\Psi x_0(\theta)\bigr) = \Psi J \Psi^{-1} (I-P) D X_{H_0}(\Psi x_0(\theta)) \Psi J x_0(\theta), \qquad \forall \theta\in S^{2k-1}.
\]
Therefore, we can choose as $Y_0$ the vector field
\[
Y(x) := \Psi J \Psi^{-1} (I-P) DX_{H_0} (x) \Psi J \Psi^{-1} x
\] 
and we find the formula
\[
f''(0) = \int_{S^{2k-1}} (\Psi\circ x_0)^* \left( - \imath_{X_{H_0}+ Y} d \imath_{X_{H_0}} \beta + \imath_{X_K} d\alpha \right),
\]
where
\[
K(x) := \frac{d H_t}{d t} (x) \Bigr|_{t=0}.
\]
Arguing as in the last part of Lemma \ref{der1}, Lemma \ref{formule} (ii) implies that the $(2k-1)$-form
\[
(\Psi\circ x_0)^* (\imath_{X_K} d\alpha)
\]
is exact, and by Stokes' theorem we find the formula
\[
f''(0) = - \int_{S^{2k-1}} (\Psi\circ x_0)^* \bigl( \imath_{X_{H_0}+ Y} d \imath_{X_{H_0}} \beta \bigr).
\]
Since $\Psi\circ x_0 = P \circ \Psi\circ x_0$, by Lemma \ref{formule} (i) we have the following chain of identities
\begin{eqnarray*}
(\Psi\circ x_0)^* \bigl( \imath_{X_{H_0}+ Y} d \imath_{X_{H_0}} \beta \bigr) &=& (P\circ \Psi\circ x_0)^* \bigl( \imath_{X_{H_0}+ Y} d \imath_{X_{H_0}} \beta \bigr) \\ &=& (\Psi\circ x_0)^* \bigl( P^* ( \imath_{X_{H_0}+ Y} d \imath_{X_{H_0}} \beta) \bigr) \\ &=&
\frac{1}{(k-1)!} (\Psi\circ x_0)^* \bigl( P^* (\Omega_k^{k-1} \wedge (\imath_{X_{H_0}+Y} d \imath_{X_{H_0}} \widehat{\Omega})) \bigr) \\ &=& \frac{1}{(k-1)!} (\Psi \circ x_0)^* \bigl( \Omega_k^{k-1} \wedge (\imath_{X_{H_0}+Y} d \imath_{X_{H_0}} \widehat{\Omega}) \bigr) \\ &=& \frac{1}{(k-1)!} x_0^* \bigl( \Psi^*(\Omega_k^{k-1}) \wedge \Psi^*(\imath_{X_{H_0}+Y} d \imath_{X_{H_0}} \widehat{\Omega}) \bigr) \\ &=& 
\frac{1}{(k-1)!} x_0^* \bigl( \Omega_k^{k-1} \wedge \Psi^*(\imath_{X_{H_0}+Y} d \imath_{X_{H_0}} \widehat{\Omega}) \bigr),
\end{eqnarray*}
where in the last equality we have used the fact that, commuting with $P$, the symplectic automorphism $\Psi$ preserves $\Omega_k$. The desired formula for $f''(0)$ follows.
\end{proof}

In order to further simplify the expression of $f''(0)$, we need to introduce some extra objects. We identify $\R^{2k}$ with $\C^k$ and we denote by $\mathrm{Gr}_1(\C^k) \cong \C\mathbb{P}^{k-1}$ the Grassmannian of complex lines in $\C^k$, equipped with its standard K\"ahler form
\[
\omega = \frac{i}{2} \partial \bar{\partial} |f|^2,
\]
where $f$ is a holomorphic local right-inverse of the projection $\C^k\setminus\{0\} \rightarrow \C \mathbb{P}^{k-1}$ (see e.g.\ \cite[Section 5.1]{jos02}). Let $\mu$ be the induced volume form on $\mathrm{Gr}_1(\C^k) \cong \C\mathbb{P}^{k-1}$, that is
\[
\mu := \frac{1}{(k-1)!} \omega^{k-1}.
\]
Then we have the following reduction lemma:

\begin{lem}
\label{red}
Let $\eta$ be a smooth one-form on $S^{2k-1}$, $k\geq 1$. Then
\[
\int_{S^{2k-1}} \Omega_k^{k-1} \wedge \eta = (k-1)! \int_{\mathrm{Gr}_1(\C^k)} \left( \int_{L\cap S^{2k-1}} \eta \right) \, \mu(L),
\] 
where the circle $L\cap S^{2k-1}$ is given the orientation induced by the complex structure of $L$.
\end{lem}

\begin{proof}
We recall that if
\[
\pi : S^{2k-1} \rightarrow \C\mathbb{P}^{k-1}
\]
is the Hopf fibration, then
\begin{equation}
\label{hopf}
\pi^* \omega = \Omega_k|_{S^{2k-1}}.
\end{equation}
Indeed, since both $\Omega_k$ and $\omega$ are invariant with respect to the action of the unitary group of $\C^k$, it is enough to check this identity at the point $(0,\dots,0,1)\in S^{2k-1}$. In the affine coordinate system
\[
\C^{k-1} \rightarrow \C\mathbb{P}^{k-1}, \qquad (\zeta_1,\dots,\zeta_{k-1}) \rightarrow [\zeta_1,\dots,\zeta_{k-1},1],
\]
the K\"ahler form $\omega$ at $[0,\dots,0,1]$ has the form
\[
\omega([0,\dots,0,1]) = \frac{i}{2} \sum_{j=1}^{k-1} d\zeta_j \wedge d \bar{\zeta}_j,
\]
(see e.g.\ \cite[Equation (5.1.5)]{jos02}). By differentiating the formula
\[
\pi (z_1,\dots,z_k) = \Bigl[\frac{z_1}{z_k}, \dots, \frac{z_{k-1}}{z_k},1 \Bigr], \qquad \forall z=(z_1,\dots,z_k) \in S^{2k-1}, \; z_k \neq 0,
\]
at $(0,\dots,0,1)$, we get that
\[
D\pi(0,\dots,0,1)[(\zeta_1,\dots,\zeta_k)] = (\zeta_1,\dots,\zeta_{k-1}), \qquad \forall (\zeta_1,\dots,\zeta_k) \in T_{(0,\dots,0,1)} S^{2k-1},
\]
from which
\[
\pi^* \omega (0,\dots,0,1) = \frac{i}{2} \sum_{j=1}^{k-1} dz_j \wedge d\bar{z}_j = \sum_{j=1}^{k-1} dp_j \wedge dq_j = \Omega_k|_{T_{(0,\dots,0,1)} S^{2k-1}},
\]
which proves (\ref{hopf}).

Let us denote by $\pi_*$ integration along the fiber, which maps $h$-forms on $S^{2k-1}$ to $(h-1)$-forms on $\C\mathbb{P}^{k-1}$. By (\ref{hopf}) and by the adjunction formula (see e.g.\ \cite[Proposition 6.15 (b)]{bt82}), we have
\[
\int_{S^{2k-1}} \Omega^{k-1}_k \wedge \eta = \int_{S^{2k-1}} \pi^* (\omega^{k-1}) \wedge \eta = \int_{\C\mathbb{P}^{k-1}} \omega^{k-1} \wedge \pi_*(\eta).
\]
Since $\eta$ is a one-form, $\pi_*(\eta)$ is the function
\[
\pi_*(\eta)(L) = \int_{\pi^{-1}(L)} \eta = \int_{L \cap S^{2k-1}} \eta, \qquad \forall L \in \mathrm{Gr}_1(\C^k) \cong \C\mathbb{P}^{k-1},
\]
and hence
\[
\int_{S^{2k-1}} \Omega^{k-1}_k \wedge \eta = (k-1)! \int_{\C\mathbb{P}^{k-1}} \left(  \int_{\pi^{-1}(L)} \eta\right) \mu(L),
\]
as claimed.
\end{proof}

Given a loop
\[
z: \R/2\pi \Z \rightarrow \R^{2n},
\]
we denote by
\[
E(z) := \frac{1}{2} \int_0^{2\pi} |z'(\theta)|^2\, d\theta
\]
its energy, and by
\[
A(z) := \frac{1}{2} \int_0^{2\pi} \Omega\bigl[ z(\theta), z'(\theta)\bigr]\, d\theta = \int_{\R/2\pi \Z} z^* \Bigl( \sum_{j=1}^n p_j \, dq_j \Bigr)
\]
the symplectic area bounded by $z$. 

\begin{lem}
\label{der2fin}
There holds
\[
f''(0) = 2 \int_{\mathrm{Gr}_1(\C^k)} \bigl( E(z_L) - A(z_L) \bigr)\, \mu (L),
\]
where $z_L:\R/2\pi \Z \rightarrow \R^{2n}$ is the loop 
\[
z_L (\theta) = \Psi^* \bigl( (I-P) X_{H_0} \bigr) \bigl( e^{\theta J} w_L\bigr),
\]
$w_L$ being a point in $L\cap S^{2k-1}$.
\end{lem} 

Here $\Psi^* ( (I-P) X_{H_0} )$ denotes the pull-back by $\Psi$ of the vector field $(I-P)X_{H_0}$, that is
\[
\Psi^* \bigl( (I-P) X_{H_0} \bigr) (x) := \Psi^{-1} (I-P) X_{H_0} (\Psi x), \qquad \forall x\in \R^{2n}.
\]

\begin{proof}
By Lemmata \ref{der2} and \ref{red}, we have
\begin{equation}
\label{quasi}
f''(0) = - \int_{\mathrm{Gr}_1(\C^k)} \left( \int_{L\cap S^{2k-1}} \Psi^*( \imath_{X_{H_0}+Y} d \imath_{X_{H_0}} \widehat{\Omega}) \right)\, \mu (L).
\end{equation}
We fix a complex line $L$ in $\C^k$ and we compute the integral of the one-form $\Psi^*(\imath_{X_{H_0}+Y} d \imath_{X_{H_0}} \widehat{\Omega})$ on $L\cap S^{2k-1}$. We parametrize the latter circle by the map
\[
w: \R/2\pi \Z \rightarrow L\cap S^{2k-1} , \qquad w(\theta) = e^{\theta J} w_L,
\]
where $w_L$ is some point in $L\cap S^{2k-1}$. Since $\widehat{\Omega}$ is a constant two-form, for every pair of vectors $u,v\in \R^{2n}$ we have
\[
d\imath_{X_{H_0}}\widehat{\Omega} (x) [u,v] = \widehat{\Omega}[DX_{H_0} (x) u,v] + \widehat{\Omega}[u,DX_{H_0}(x) v], \quad \forall x\in \R^{2n},
\]
from which we infer that for every vector $u\in \R^{2n}$ there holds
\[
\imath_{X_{H_0}+Y} d\imath_{X_{H_0}} \widehat{\Omega} (x) [u]= \widehat{\Omega}[ DX_{H_0} (x) (X_{H_0} (x) + Y(x)),u] + \widehat{\Omega}[X_{H_0}(x)+Y(x), DX_{H_0} (x) u], 
\]
and 
\begin{eqnarray*}
\Psi^* (\imath_{X_{H_0}+Y} d\imath_{X_{H_0}} \widehat{\Omega}) (x) [u]= & &\widehat{\Omega}[ DX_{H_0} (\Psi x) (X_{H_0} (\Psi x) + Y(\Psi x)),\Psi u] \\ &+& \widehat{\Omega}[X_{H_0}(\Psi x)+Y(\Psi x), DX_{H_0} (\Psi x) \Psi u], 
\end{eqnarray*}
for every $x\in \R^{2n}$. In particular, since $\Psi w' = \Psi Jw$ belongs to the kernel of $\widehat{\Omega}$,
\[
\Psi^* (\imath_{X_{H_0}+Y} d\imath_{X_{H_0}} \widehat{\Omega}) (w) [w'] = \widehat{\Omega}[X_{H_0}(\Psi w) + Y(\Psi w), DX_{H_0}(\Psi w) \Psi w'], \qquad \mbox{on } \R/2\pi \Z.
\]
We compute the two terms
\[
\widehat{\Omega}[X_{H_0}(\Psi w), DX_{H_0}(\Psi w) \Psi w'], \qquad \widehat{\Omega}[Y(\Psi w), DX_{H_0}(\Psi w) \Psi w']
\]
separately. 

By the identity
\[
\widehat{\Omega} [u,v] = \Omega \bigl[ (I-P) u , (I-P) v \bigr], \qquad \forall u,v\in \R^{2n},
\]
the first term equals
\begin{eqnarray*}
\widehat{\Omega}[X_{H_0}(\Psi w), DX_{H_0}(\Psi w) \Psi w'] &=& \Omega[ (I-P) X_{H_0}(\Psi w) , (I-P) DX_{H_0} (\Psi w) \Psi w'] \\ &=& \Omega[ \tilde{z}_L , \tilde{z}_L' ],
\end{eqnarray*}
where
\[
\tilde{z}_L(\theta) := (I-P) X_{H_0} \bigl( \Psi w(\theta) \bigr), \qquad \forall \theta\in \R/2\pi \Z.
\]
By the identity
\[
\widehat{\Omega} [u,v] = J (I-P) u \cdot (I-P) v, \qquad \forall u,v\in \R^{2n},
\]
and by the explicit expression of $Y$ (see Lemma \ref{der2}), the second term equals
\begin{eqnarray*}
&&\widehat{\Omega}[Y(\Psi w), DX_{H_0}(\Psi w) \Psi w'] \\
&& \quad = J (I-P) \Psi J \Psi^{-1} (I-P) DX_{H_0}(\Psi w) \Psi w' \cdot (I-P) DX_{H_0} (\Psi w) \Psi w',
\end{eqnarray*}
and, using the identities $\Psi J = J (\Psi^{-1})^T$, $[J,P]=0$ and $J^2=-I$, we get
\begin{eqnarray*}
&&\widehat{\Omega}[Y(\Psi w), DX_{H_0}(\Psi w) \Psi w'] \\
&& \quad = J (I-P) J (\Psi^{-1})^T \Psi^{-1} (I-P) DX_{H_0}(\Psi w) \Psi w' \cdot (I-P) DX_{H_0} (\Psi w) \Psi w' \\
&& \quad = - \bigl| \Psi^{-1} (I-P) DX_{H_0} (\Psi w) \Psi w' \bigr|^2 = - |z_L'|^2,
\end{eqnarray*}
where
\[
z_L(\theta) := \Psi^{-1} \tilde{z}_L(\theta) = 
\Psi^{-1} (I-P) X_{H_0} \bigl( \Psi w(\theta) \bigr), \qquad \forall \theta\in \R/2\pi \Z.
\]
These computations show that
\[
\Psi^* (\imath_{X_{H_0}+Y} d\imath_{X_{H_0}} \widehat{\Omega}) (w) [w'] = \Omega[z_L,z_L'] - |z_L'|^2,
\]
where we have used the equality $\Omega[\tilde{z}_L,\tilde{z}_L']=\Omega[z_L,z_L']$, which is due to the fact that $\Psi$ is symplectic. By integrating this identity over $\R/2\pi \Z$, we conclude that
\[
\int_{L \cap S^{2k-1}} \Psi^* (\imath_{X_{H_0}+Y} d\imath_{X_{H_0}} \widehat{\Omega})= \int_0^{2\pi} \bigl(  \Omega[z_L,z_L'] - |z_L'(\theta)|^2 ]\bigr) \, dt = 2 A(z_L) - 2 E(z_L),
\]
and the thesis follows from (\ref{quasi}).
\end{proof}

\section{Local non-squeezing}
\label{7sec}

Most of the computation having been performed in the previous section, we can finally prove Theorem \ref{expa}:

\begin{thm}[Second order expansion]
Let $\{\phi_t\}_{t\in [0,1]}$ be a smooth one-parameter family of symplectic diffeomorphisms of $\R^{2n}$ into itself such that $\phi_0 = \Phi$ is linear and let $\{H_t\}_{t\in [0,1]}$ be its generating path of Hamiltonians. Let $P$ be the orthogonal projector onto a complex linear subspace $V\subset \R^{2n}$ of real dimension $2k$, with $1\leq k \leq n$, and assume that also $\Phi^T V$ is a complex subspace. Then
\[
\mathrm{vol}_{2k}\bigl(P \phi_t(B)\bigr) = \omega_{2k} + C \, t^2 + O(t^3), \qquad \mbox{for } t\rightarrow 0,
\]
the number $C$ being defined as
\[
C= C(H_0,\Phi) := \int_{\mathrm{Gr}_1(\Phi^T V)} \bigl( E(\zeta_L) - A(\zeta_L) \bigr)\, \mu (L),
\]
where $\zeta_L:\R/2\pi \Z\rightarrow \R^{2n}$ is the loop
\[
\zeta_L(\theta) = \Phi^* \bigr( (I-P) X_{H_0} \bigr) \bigl( e^{\theta J} \xi_L\bigr),
\]
$\xi_L$ being an arbitrary unit vector in $L\in \mathrm{Gr}_1(\Phi^T V)$.
\end{thm}

\begin{proof}
Up to the left-composition of $\phi_t$ by an orthogonal and symplectic automorphism of $\R^{2n}$ - that is unitary on $\C^n\cong \R^{2n}$ - we may assume that $V=\R^{2k}$, so that $P$ is the projector which was considered in the last two sections. By Theorem \ref{lns} and Lemmata \ref{der1} and \ref{der2fin}, the function
\[
f(t) = \mathrm{vol}_{2k}\bigl(P \phi_t(B)\bigr) = \mathrm{vol}_{2k}\bigl(P \psi_t(B)\bigr)
\]
satisfies
\[
f(0) = \omega_{2k}, \quad f'(0) = 0, \quad f''(0) = 2 \int_{\mathrm{Gr}_1(V)} \bigl( E(z_M) - A(z_M) \bigr)\, \mu (M),
\]
where
\[
z_M (\theta) = \Psi^{-1} (I-P) X_{H_0} \bigl( \Psi e^{\theta J} w_M \bigr), \qquad \forall \theta\in \R/2\pi \Z, \quad w_M \in S^{2k-1}\cap M, \quad M \in \mathrm{Gr}_1(V),
\]
and $\Psi = \Phi U$, $U$ being a unitary automorphism of $\R^{2n}\cong \C^n$ which maps $\R^{2k} \cong \C^k$ onto $\Phi^T \R^{2k}$. We can express the loop $z_M$ in terms of $\Phi$ and $U$ as
\[
z_M(\theta) = U^{-1} \Phi^{-1} (I-P)X_{H_0} \bigl( \Phi U e^{\theta J} w_M \bigr),
\]
and we notice that
\[
E(z_M) = E(\hat{z}_M), \qquad A(z_M) = A(\hat{z}_M),
\]
where
\[
\hat{z}_M(\theta) := U z_M(\theta) = \Phi^{-1} (I-P)X_{H_0} \bigl( \Phi U e^{\theta J} w_M \bigr).
\]
So we may replace $z_M$ by $\hat{z}_M$ in the above formula for $f''(0)$.
If we set, for every $L\in \mathrm{Gr}_1(\Phi^T V)$, 
\[
\xi_{L} := U w_{U^{-1} L}, \qquad \zeta_{L}(\theta) := \Phi^{-1} (I-P) X_{H_0} \bigl( \Phi e^{\theta J} \xi_L\bigr),
\]
we have that $\xi_L$ belongs to $L\cap S^{2k-1}$ and, since $U$ commutes with $e^{\theta J}$,
\[
\zeta_L = \hat{z}_{U^{-1} L}.
\]
From the invariance properties of the standard K\"ahler form of the complex Grassmannians with respect to unitary transformations, we may apply the change of variable $M=U^{-1} L$, $L\in \mathrm{Gr}_1(\Phi^T V)$, and get the identity
\begin{eqnarray*}
f''(0) &=& 2 \int_{\mathrm{Gr}_1(\Phi^T V)} \bigl(E(\hat{z}_{U^{-1} L}) - A(\hat{z}_{U^{-1} L}) \bigr)\, \mu (L) \\ &=& 2 \int_{\mathrm{Gr}_1(\Phi^T V)} \bigl(E(\zeta_{L}) - A(z_{L}) \bigr)\, \mu (L) =: 2 \, C (H_0,\Phi).
\end{eqnarray*}
The conclusion follows from Taylor's formula.
\end{proof}

We recall the following well-known area-energy inequality, whose simple proof is included for sake of completeness:

\begin{lem}
\label{poi}
For every smooth loop $z: \R/2\pi \Z \rightarrow \R^{2n}$, there holds
\[
A(z) \leq E(z),
\]
and the equality holds if and only if $z$ has the form
\begin{equation}
\label{forma}
z(\theta) = \frac{1}{2} \bigl( z(0) + z(\pi) \bigr) + \frac{1}{2} e^{\theta J} \bigl( z(0) - z(\pi) \bigr).
\end{equation}
\end{lem}

\begin{proof}
If the loop $z$ has the Fourier expansion
\[
z(\theta) = \sum_{k\in \Z} e^{k\theta J} z_k, \qquad \mbox{with } z_k \in \R^{2n},
\]
we have
\[
E(z) = \pi \sum_{k\in \Z} k^2 |z_k|^2, \qquad
A(z) = \pi \sum_{k\in \Z} k  |z_k|^2.
\]
Therefore,
\begin{eqnarray*}
 E(z) - A(z) = \pi \sum_{k\in \Z} k^2 |z_k|^2 -   \pi \sum_{k\in \Z} k  |z_k|^2  = \pi \sum_{k\in \Z} k(k-1)  |z_k|^2 \geq 0,
\end{eqnarray*}
which proves that $A(z)\leq E(z)$ and that the equality holds if and only if $z_k=0$ for every $k$ different from $0$ and $1$. In the latter case,
\[
z(\theta) = z_0 + e^{\theta J} z_1,
\]
which can be written as (\ref{forma}) by solving the linear system
\[
z(0) = z_0 + z_1, \qquad z(\pi) = z_0 - z_1,
\]
for $z_0$ and $z_1$.
\end{proof}

Together with Lemma \ref{poi}, Theorem \ref{expa} has the following consequence:

\begin{cor*}[Strict local non-squeezing]
Let $\phi_t$, $\Phi$, $H_t$, $P$ and $V$ be as in Theorem \ref{expa}, with $1\leq k \leq n-1$. Assume that the vector field $Z:=\Phi^* ((I-P)X_{H_0})$ does not satisfy the symmetry condition
\begin{equation}
\label{simm2}
Z \left( e^{\theta J} x \right) =  \frac{1}{2} \bigl( Z(x)+Z(-x)\bigr)  + \frac{1}{2} e^{\theta J} \bigl( Z  (x) - Z(-x) \bigr) , \quad \forall x\in \partial B \cap \Phi^T V, \; \forall \theta \in \R/2\pi \Z.
\end{equation}
Then there exists $t_0>0$ such that
\begin{equation}
\label{strict2}
\mathrm{vol}_{2k} \bigl(P\phi_t (B) \bigr) > \omega_{2k} , \qquad \mbox{for every } 0 < t < t_0.
\end{equation}
\end{cor*}

\begin{proof}
By Theorem \ref{expa},
\begin{equation}
\label{espa}
\mathrm{vol}_{2k} \bigl(P\phi_t (B) \bigr) = \omega_{2k} + C(H_0,\Phi) \, t^2 + O(t^3), \qquad \mbox{for } t\rightarrow 0.
\end{equation}
Lemma \ref{poi} implies that $C(H_0,\Phi)\geq 0$ and that it equals zero if and only if
\[
\zeta_L(\theta) = \frac{1}{2} \bigl( \zeta_L(0) + \zeta_L(\pi) \bigr) + \frac{1}{2} e^{\theta J} \bigl( \zeta_L(0) - \zeta_L(\pi) \bigr),
\]
for every $L\in \mathrm{Gr}_1(\Phi^T V)$. Since $\zeta_L(\theta) = Z(e^{\theta J} \xi_L)$, the latter condition is equivalent to
\[
Z \left( e^{\theta J} \xi_L \right) =  \frac{1}{2} \bigl( Z(\xi_L)+Z(-\xi_L)\bigr)  + \frac{1}{2} e^{\theta J} \bigl( Z  (\xi_L) - Z(-\xi_L) \bigr),
\]
for every $L\in \mathrm{Gr}_1(\Phi^T V)$. Since $\xi_L$ is an arbitrary point in $L \cap \partial B$, the last condition is equivalent to (\ref{simm2}). So, when (\ref{simm2}) does not hold, $C(H_0,\Phi)$ is positive and (\ref{strict2}) follows from (\ref{espa}).
\end{proof}

\begin{rem*}
The above corollary implies a strict non-squeezing inequality also for the first formulation of the local problem. Indeed,
if $\phi:\R^{2n} \rightarrow \R^{2n}$ is a symplectic diffeomorphism and $\{\phi_t\}_{t\in [0,1]}$ is the path defined by (\ref{path}), then the vector field $Z$ appearing above is easily shown to be the quadratic map
\[
Z(x) = \frac{1}{2} D\phi(0)^{-1} (I-P) D^2 \phi(0)[x,x].
\]
In this case, condition (\ref{simm2}) is equivalent to
\begin{equation}
\label{simm3}
(I-P) D^2 \phi(0) [Jx,x] = 0, \qquad \forall x\in D\phi(0)^T V.
\end{equation}
We conclude that if (\ref{simm3}) does not hold then there exists $R_0>0$ such that
\[
\mathrm{vol}_{2k} \bigl( P \phi (B_R) \bigr) > \omega_{2k} R^{2k},
\]
for every $0<R<R_0$.
\end{rem*}

We conclude the paper by discussing the relationship between the middle-dimensional non-squeezing question which is addressed here and the following conjecture of C.~Viterbo's \cite[Section 5]{vit00}: if 
\[
c: \{\mbox{open subsets of }\R^{2n}\} \rightarrow [0,+\infty]
\]
is a symplectic capacity (see \cite[Chapter 2]{hz94} for the list of axioms which define such an invariant) and if $A\subset \R^{2n}$ is an open bounded convex set, then
\begin{equation}
\label{vite}
\frac{c(A)}{\pi} \leq \left( \frac{\mathrm{vol}_{2n}(A)}{\omega_{2n}} \right)^{1/n}.
\end{equation}
Let $c$ be a symplectic capacity which satisfies the following additional property: $c(PA)\geq c(A)$ whenever $P$ is the projector onto a symplectic subspace $V$ along its symplectic orthogonal (here $c(PA)$ is the capacity of $PA$ as an open subset of $V$). For instance, the cylindrical symplectic capacity
\[
c_Z(A) := \inf \set{ \pi r^2}{A \mbox{ embeds symplectically into the cylinder } B_r^2 \times \R^{2n-2}}
\]
has this property. 

Now let $P$ be the orthogonal projector onto a complex subspace of real dimension $2k$, and let $\varphi:B\rightarrow \R^{2n}$ be a symplectomorphism such that $P \varphi(B)$ is convex (for instance, because $\varphi$ is $C^1$-close to a linear map). If Viterbo's conjecture holds true for the symplectic capacity $c$ in dimension $2k$, then
\[
1 = \frac{c(B)}{\pi} = \frac{c\bigl( \varphi(B) \bigr)}{\pi} \leq \frac{c\bigl( P \varphi(B) \bigr)}{\pi} \leq \left( \frac{\mathrm{vol}_{2k}\bigl(P \varphi(B)\bigr)}{\omega_{2k}} \right)^{1/k},
\]
and hence
\[
\mathrm{vol}_{2k}(P \varphi(B)) \geq \omega_{2k}.
\]
Therefore, Viterbo's conjecture implies the middle dimensional non-squeezing inequality which is considered in this paper, whenever the shadow of the symplectic ball is convex. So far, the best result concerning Viterbo's conjecture for arbitrary convex domains is due to S.\ Artstein-Avidan, V.\ Milman and Y.\ Ostrover \cite{avo08}  and states that the inequality
\[
\frac{c(A)}{\pi} \leq C \left( \frac{\mathrm{vol}_{2n}(A)}{\omega_{2n}} \right)^{1/n}
\]
holds with a constant $C>1$ which does not depend on $n$. Very recently, J.\ C.\ \'Alvarez Paiva and F.\ Balacheff have proved that the inequality with the sharp constant $C=1$ holds locally for a very large family of one-parameter deformations of the ball, in the case where $c(A)$ is the minimum of the symplectic area bounded by closed characteristics on $\partial A$ (see \cite{apb11}, and in particular \cite{apb12}). This result implies that the conclusion of the above Corollary holds for a much larger class of paths of symplectic diffeomorphisms starting at a linear one. More details about this will appear elsewhere. 


\begin{thebibliography}{AAMO08}

\bibitem[AAMO08]{avo08}
S.~Artstein-Avidan, V.~Milman, and Y.~Ostrover, \emph{The {M}-ellipsoid,
  symplectic capacities and volume}, Comm. Math. Helv. \textbf{83} (2008),
  359--369.

\bibitem[Alm]{alm}
F.~J. Almgren, \emph{The theory of varifolds - {A} calculus of variations in
  the large for the k-dimensional area integrated}, manuscript available in the
  {P}rinceton mathematics library.

\bibitem[APB11]{apb11}
J.~C. \'{A}lvarez Paiva and F.~Balacheff, \emph{Contact geometry and
  isosystolic inequalities}, {\tt arXiv:1109.4253v1 [math.SG]}, 2011.

\bibitem[APB12]{apb12}
J.~C. \'{A}lvarez Paiva and F.~Balacheff, \emph{On a conjecture of {V}iterbo's}, (in preparation), 2012.

\bibitem[BH11]{bh11}
O.~Buse and R.~Hind, \emph{Symplectic embeddings of ellipsoids in dimension
  greater than four}, Geometry and Topology \textbf{15} (2011), 2091--2110.

\bibitem[BT82]{bt82}
R.~Bott and L.~W. Tu, \emph{Differential forms in algebraic topology},
  Springer, 1982.

\bibitem[CI11]{cm11}
T.~H. Colding and W.~P.~Minicozzi II, \emph{A course in minimal surfaces},
  Graduate Studies in Mathematics, American Mathematical Society, 2011.

\bibitem[EG91]{eg91}
Y.~Eliashberg and M.~Gromov, \emph{Convex symplectic manifolds}, Several
  complex variables and complex geometry, Part 2 (Santa Cruz, CA, 1989), Proc.
  Sympos. Pure Math., vol. 52 part 2, Amer. Math. Soc., Providence, RI, 1991,
  pp.~135--162.

\bibitem[EH89]{eh89}
I.~Ekeland and H.~Hofer, \emph{Symplectic topology and {H}amiltonian dynamics},
  Math. Z. \textbf{200} (1989), 355--378.

\bibitem[Fed69]{fed69}
H.~Federer, \emph{Geometric measure theory}, Springer, 1969.

\bibitem[Gro83]{gro83}
M.~Gromov, \emph{Filling {R}iemannian manifolds}, J. Differential Geom.
  \textbf{18} (1983), 1--147.

\bibitem[Gro85]{gro85}
M.~Gromov, \emph{Pseudo holomorphic curves in symplectic manifolds}, Invent.
  Math. \textbf{82} (1985), 307--347.

\bibitem[Gro03]{gro03}
M.~Gromov, \emph{Isoperimetry of waists and concentration of maps}, Geom. Funct.
  Anal. \textbf{13} (2003), 178--215.

\bibitem[Gut08]{gut08}
L.~Guth, \emph{Symplectic embeddings of polydisks}, Invent. Math. \textbf{172}
  (2008), 477--489.

\bibitem[Gut09]{gut09}
L.~Guth, \emph{Minimax problems related to cup powers and {S}teenrod squares},
  Geom. Funct. Anal. \textbf{18} (2009), 1917--1987.

\bibitem[HK10]{hk10}
R.~Hind and E.~Kerman, \emph{New obstructions to symplectic embeddings}, {\tt
  arXiv :0906.4296v2 [math.SG]}, 2010.

\bibitem[Hut10]{hut10}
M.~Hutchings, \emph{The embedded contact homology and its applications},
  Proceedings of the {I}nternational {C}ongress of {M}athematicians (Hyderabad,
  India), 2010.

\bibitem[HZ94]{hz94}
H.~Hofer and E.~Zehnder, \emph{Symplectic invariants and {H}amiltonian
  dynamics}, Birkh\"auser, Basel, 1994.

\bibitem[Jos02]{jos02}
J.~Jost, \emph{Riemannian geometry and geometric analysis}, third ed.,
  Springer, Berlin, 2002.

\bibitem[Kuk95]{kuk95}
S.~B. Kuksin, \emph{Infinite-dimensional symplectic capacities and a squeezing
  theorem for {H}amiltonian {PDE}'s}, Commun. Math. Phys. \textbf{167} (1995),
  531--552.

\bibitem[LM95]{lm95}
F.~Lalonde and D.~McDuff, \emph{Local non-squeezing theorems and stability},
  Geom. Funct. Anal. \textbf{5} (1995), 364--386.

\bibitem[McD11]{mcd11}
D.~McDuff, \emph{The {H}ofer conjecture on embedding symplectic ellipsoids}, J.
  Differential Geom. \textbf{88} (2011), 519--532.

\bibitem[Mem11]{mem11}
Y.~Memarian, \emph{On {G}romov's waist of the sphere theorem}, J. Topol. Anal.
  \textbf{3} (2011), 7--36.

\bibitem[Mos65]{mos65}
J.~Moser, \emph{On the volume elements of a manifold}, Trans. Amer. Math. Soc.
  \textbf{120} (1965), 286--294.

\bibitem[MS98]{ms98}
D.~McDuff and D.~Salamon, \emph{Introduction to symplectic topology}, second
  ed., Oxford Mathematical Monographs, The Clarendon Press Oxford University
  Press, New York, 1998.

\bibitem[MS12]{ms12}
D.~McDuff and F.~Schlenk, \emph{The embedding capacity of 4-dimensional
  symplectic ellipsoids}, Ann. of Math. \textbf{175} (2012), 1191--1282.

\bibitem[Sch05]{sch05}
F.~Schlenk, \emph{Embedding problems in symplectic geometry}, de Gruyter
  Expositions in Mathematics, vol.~40, Walter de Gruyter, 2005.

\bibitem[Vit89]{vit89}
C.~Viterbo, \emph{Capacit\'e symplectiques et applications}, Ast\'erisque
  \textbf{177-178} (1989), no.~714, S\'eminaire Bourbaki 41\'eme ann\'ee,
  345--362.

\bibitem[Vit92]{vit92}
C.~Viterbo, \emph{Symplectic topology as the geometry of generating functions},
  Math. Ann. \textbf{292} (1992), 685--710.

\bibitem[Vit00]{vit00}
C.~Viterbo, \emph{Metric and isoperimetric problems in symplectic geometry}, J.
  Amer. Math. Soc. \textbf{13} (2000), 411--431.

\bibitem[Wan03]{wan03}
M.-T. Wang, \emph{The {D}irichlet problem for the minimal surface system in
  arbitrary dimensions and codimensions}, Comm. Pure Appl. Math. \textbf{57}
  (2003), 267--281.

\end{thebibliography}

\providecommand{\bysame}{\leavevmode\hbox to3em{\hrulefill}\thinspace}
\providecommand{\MR}{\relax\ifhmode\unskip\space\fi MR }
\providecommand{\MRhref}[2]{%
  \href{http://www.ams.org/mathscinet-getitem?mr=#1}{#2}
}
\providecommand{\href}[2]{#2}

\end{document}